\documentclass[a4paper,12pt]{amsart}
\usepackage{times,fullpage}
\usepackage{latexsym,amscd,amssymb,url}
\usepackage[latin1]{inputenc} 	
\usepackage[T1]{fontenc}       
\usepackage{ae}									
\usepackage[arrow, matrix, curve]{xy}  
\usepackage{graphicx}
\usepackage{MnSymbol}  

\newtheorem{theorem}{Theorem}

\newtheorem*{claim}{Claim}
\newtheorem{corollary}[theorem]{Corollary}
\newtheorem{definition}[theorem]{Definition}
\newtheorem{lemma}[theorem]{Lemma}
\newtheorem{proposition}[theorem]{Proposition}
\newtheorem{remark}[theorem]{Remark}

\numberwithin{equation}{section}

\newcommand{\Z}{\mathbb Z}

\newcommand{\C}{\mathbb C}
\newcommand{\N}{\mathbb N}
\newcommand{\R}{\mathbb R}
\newcommand{\CP}{\mathbb P}

\newcommand{\Fix}{\operatorname{Fix}}
\newcommand{\Aut}{\operatorname{Aut}}
\newcommand{\sign}{\operatorname{sign}}
\newcommand{\Sym}{\operatorname{Sym}}
\newcommand{\id}{\operatorname{id}}

\newcommand{\GL}{\operatorname{GL}}

\begin{document}

\title{On the construction problem for Hodge numbers}
\author{Stefan Schreieder} 
\address{Max-Planck-Institut für Mathematik, Vivatsgasse 7, 53111 Bonn, Germany}							
\address{Mathematisches Institut, Universität Bonn, Endenicher Allee 60, 53115 Bonn, Germany} 
\email{schreied@math.uni-bonn.de}%

\date{April 25, 2014}


\begin{abstract}
For any symmetric collection $(h^{p,q})_{p+q=k}$ of natural numbers, we construct a smooth complex projective variety $X$ whose weight $k$ Hodge structure has Hodge numbers $h^{p,q}(X)=h^{p,q}$; if $k=2m$ is even, then we have to impose that $h^{m,m}$ is bigger than some quadratic bound in $m$.
Combining these results for different weights, we solve the construction problem for the truncated Hodge diamond under two additional assumptions.
Our results lead to a complete classification of all nontrivial dominations among Hodge numbers of Kähler manifolds.
\end{abstract}

\maketitle


\section{Introduction}
For a Kähler manifold $X$, Hodge theory yields an isomorphism
\begin{align} \label{eq:Hodgestructure}
H^k(X,\C)\cong \bigoplus_{p+q=k}H^q(X,\Omega_X^p) \ .
\end{align}
As a refinement of the Betti numbers of $X$, one therefore defines the $(p,q)$-th Hodge number $h^{p,q}(X)$ of $X$ to be the dimension of $H^q(X,\Omega_X^p)$.
This way one can associate to each $n$-dimensional Kähler manifold $X$ its collection of Hodge numbers $h^{p,q}(X)$ with $0\leq p,q \leq n$.
Complex conjugation and Serre duality show that such a collection of Hodge numbers $(h^{p,q})_{p,q}$ in dimension $n$ needs to satisfy the Hodge symmetries
\begin{align} \label{eq:Hsym}
h^{p,q}=h^{q,p}=h^{n-p,n-q} \ .
\end{align}
Moreover, as a consequence of the Hard Lefschetz Theorem, the Lefschetz conditions
\begin{align}\label{eq:Lineq}
h^{p,q}\geq h^{p-1,q-1} \ \ \ \text{for all}\ \ \ p+q\leq n 
\end{align}
hold.
Given these classical results, the construction problem for Hodge numbers asks which collections of natural numbers $(h^{p,q})_{p,q}$, satisfying (\ref{eq:Hsym}) and (\ref{eq:Lineq}), actually arise as Hodge numbers of some $n$-dimensional Kähler manifold.
In his survey article on the construction problem in Kähler geometry \cite{Si}, C.\ Simpson explains our lack of knowledge on this problem.
Indeed, even weak versions where instead of all Hodge numbers one only considers small subcollections of them are wide open; for some partial results in dimensions two and three we refer to \cite{surfaces,chang,hunt}. 

This paper provides three main results on the above construction problem in the category of smooth complex projective varieties, which is stronger than allowing arbitrary Kähler manifolds.
We present them in the following three subsections respectively.
 
\subsection{The construction problem for weight $k$ Hodge structures}
It follows from Griffiths transversality that a general integral weight $k$ ($k\geq 2$) Hodge structure (not of K3 type) cannot be realized by a smooth complex projective variety, see \cite[Remark 10.20]{voisin1}.
This might lead to the expectation that general weight $k$ Hodge numbers can also not be realized by smooth complex projective varieties.
Our first result shows that this expectation is wrong. 
This answers a question in \cite{Si}. 

\begin{theorem} \label{thm:main2}
Fix $k\geq 1$ and let $(h^{p,q})_{p+q=k}$ be a symmetric collection of natural numbers.
If $k=2m$ is even, we assume 
	\[
	h^{m,m}\geq m\cdot \left\lfloor (m+3)/2\right\rfloor+\left\lfloor m/2\right\rfloor^2 \ .
	\]
Then in each dimension $\geq k+1$ there exists a smooth complex projective variety whose Hodge structure of weight $k$ realizes the given Hodge numbers.
\end{theorem}

The examples which realize given weight $k$ Hodge numbers in the above theorem have dimension $\geq k+1$.
However, if we assume that the outer Hodge number $h^{k,0}$ vanishes and that the remaining Hodge numbers are even, then we can prove a version of Theorem \ref{thm:main2} also in dimension $k$, see Corollary \ref{cor:main2} in Section \ref{sec:main2}.

Since any smooth complex projective variety contains a hyperplane class, it is clear that some kind of bound on $h^{m,m}$ in Theorem \ref{thm:main2} is necessary.
For $m=1$, for instance, the bound provided by the above Theorem is $h^{1,1}\geq 2$.
In Section \ref{sec:deg2coho} we will show that in fact the optimal bound $h^{1,1}\geq 1$ can be reached.
That is, we will show (Theorem \ref{thm:main3}) that any natural numbers $h^{2,0}$ and $h^{1,1}$ with $h^{1,1}\geq 1$ can be realised as weight two Hodge numbers of some smooth complex projective variety. 
For $m\geq 2$, we do not know whether the bound on $h^{m,m}$ in Theorem \ref{thm:main2} is optimal or not.

\subsection{The construction problem for the truncated Hodge diamond}
Given Theorem \ref{thm:main2} one is tempted to ask for solutions to the construction problem for collections of Hodge numbers which do not necessarily correspond to a single cohomology group.
In order to explain our result on this problem, we introduce the following notion:
An $n$-dimensional formal Hodge diamond is a table
\begin{align} \label{eq:diamond}
\begin{tabular}{ l c c c c c c c r }
    & & &  & $h^{n,n}$ & &  & &\\
    & & &  $h^{n,n-1}$ &  & $h^{n-1,n}$ & & & \\
    &  & $h^{n,n-2}$ &  & $h^{n-1,n-1}$ &  & $h^{n-2,n}$ & & \\
    &  \reflectbox{$\ddots$} &  &  & $\vdots$ &  &  & $\ddots$ & \\
 $h^{n,0}$ & & & &  & &   & & $h^{0,n}$ \\
    & $\ddots$  &  &  & $\vdots$ &  &  &\reflectbox{$\ddots$}  & \\
    &  & $h^{2,0}$ &  & $h^{1,1}$ &  & $h^{0,2}$ & & \\
    & & &  $h^{1,0}$ &  & $h^{0,1}$ & & & \\
    & & &  & $h^{0,0}$ & &  & &\\
\end{tabular}
\end{align}
of natural numbers $h^{p,q}$, satisfying the Hodge symmetries (\ref{eq:Hsym}), the Lefschetz conditions (\ref{eq:Lineq}) and the connectivity condition $h^{0,0}=h^{n,n}=1$.
The $h^{p,q}$ are referred to as Hodge numbers and the sum over all $h^{p,q}$ with $p+q=k$ as $k$-th Betti number $b_k$ of this formal diamond; the vector $(b_0,\ldots,b_{2n})$ is called a vector of formal Betti numbers.
Finally, for $p+q\leq n$, the primitive $(p,q)$-th Hodge number of the above diamond is defined via
\[
l^{p,q}:=h^{p,q}-h^{p-1,q-1} \ .
\]

\begin{definition}
A truncated $n$-dimensional formal Hodge diamond is a formal Hodge diamond (\ref{eq:diamond}) as above where the horizontal middle axis, i.e.\ the row of Hodge numbers $h^{p,q}$ with $p+q=n$, is omitted.
\end{definition}

We note that for a Kähler manifold $X$ its truncated Hodge diamond together with all holomorphic Euler characteristics $\chi(X,\Omega_X^p)$, where $p=0,\ldots ,\left\lfloor n/2\right\rfloor$, is equivalent to giving the whole Hodge diamond.
It is shown in \cite{K&S} that a linear combination of Hodge numbers can be expressed in terms of Chern numbers if and only if it is a linear combination of these Euler characteristics. 
Therefore, the Hodge numbers of the truncated Hodge diamond form a complement to the space of Hodge numbers which are determined by Chern numbers, cf.\ \cite{K&S} where the Hodge numbers in dimension $n$ are regarded as linear forms on the weight $n$ part of a certain graded ring.

Our second main result solves the construction problem for the truncated Hodge diamond under two additional assumptions:

\begin{theorem} \label{thm:main1}
Suppose we are given a truncated $n$-dimensional formal Hodge diamond whose Hodge numbers $h^{p,q}$ satisfy the following two additional assumptions:
\begin{enumerate}
	\item For $p< n/2$, the primitive Hodge numbers $l^{p,p}$ satisfy \label{cond1:thm:main1}
	\[
	l^{p,p}\geq p\cdot (n^2-2n+5)/4 \ .
	\]
	\item The outer Hodge numbers $h^{k,0}$ vanish either for all $k=1,\ldots , n-3$, or for all $k\neq k_0$ for some $k_0\in\left\{1,\ldots ,n-1\right\}$.
\end{enumerate}
Then there exists an $n$-dimensional smooth complex projective variety whose truncated Hodge diamond coincides with the given one.
\end{theorem}

Theorem \ref{thm:main1} has several important consequences.
For instance, for the union of $h^{n-2,0}$ and $h^{n-1,0}$ with the collection of all Hodge numbers which neither lie on the boundary, nor on the horizontal or vertical middle axis of (\ref{eq:diamond}), the construction problem is solvable without any additional assumptions.
That is, the corresponding subcollection of any $n$-dimensional formal Hodge diamond can be realized by a smooth complex projective variety.
The number of Hodge numbers we omit in this statement from the whole diamond (\ref{eq:diamond}) grows linearly in $n$, whereas the number of all entries of (\ref{eq:diamond}) grows quadratically in $n$.
In this sense, Theorem \ref{thm:main1} yields very good results on the construction problem in high dimensions.

Theorem \ref{thm:main1} deals with Hodge structures of different weights simultaneously.
This enables us to extract from it results on the construction problem for Betti numbers.
Indeed, the following corollary rephrases Theorem \ref{thm:main1} in terms of Betti numbers.

\begin{corollary} \label{cor:betti} 
Given a vector $(b_0,\ldots, b_{2n})$ of formal Betti numbers with
\[
b_{2k}-b_{2k-2}\geq k\cdot (n^2-2n+5)/8 \ \ \text{for all $k< n/2$.}
\]
Then there exists an $n$-dimensional smooth complex projective variety $X$ with $b_k(X)=b_k$ for all $k\neq n$.
\end{corollary}

This corollary says for instance that in even dimensions, the construction problem for the odd Betti numbers is solvable without any additional assumptions.

\subsection{Universal inequalities and Koll\'ar--Simpson's domination relation}
Following Koll\'ar--Simpson \cite[p.\ 9]{Si}, we say that a Hodge number $h^{r,s}$ dominates $h^{p,q}$ in dimension $n$, if there exist positive constants $c_1,c_2\in \R_{>0}$ such that for all $n$-dimensional smooth complex projective varieties $X$, the following holds:
\begin{equation} \label{eq:domination}
c_1\cdot h^{r,s}(X)+c_2\geq h^{p,q}(X) \ .
\end{equation}
Moreover, such a domination is called nontrivial if $(0,0)\neq(p,q)\neq (n,n)$, and if (\ref{eq:domination}) does not follow from the Hodge symmetries (\ref{eq:Hsym}) and the Lefschetz conditions (\ref{eq:Lineq}).

In \cite{Si} it is speculated that the middle Hodge numbers should probably dominate the outer ones.
In our third main theorem of this paper, we classify all nontrivial dominations among Hodge numbers in any given dimension.
As a result we see that the above speculation is accurate precisely in dimension two.

\begin{theorem}\label{thm:domination}
The Hodge number $h^{1,1}$ dominates $h^{2,0}$ nontrivially in dimension two and this is the only nontrivial domination in dimension two.
Moreover, there are no nontrivial dominations among Hodge numbers in any dimension different from two.
\end{theorem}

Firstly, using the classification of surfaces and the Bogomolov--Miyaoka--Yau inequality, we will prove in  Section \ref{sec:domination} (Proposition \ref{prop:Hodgeineq}) that 
\[
h^{1,1}(X)>h^{2,0}(X) 
\]
holds for all Kähler surfaces $X$.
That is, the middle degree Hodge number $h^{1,1}$ indeed dominates $h^{2,0}$ nontrivially in dimension two.

Secondly, in addition to Theorem \ref{thm:main1}, the proof of Theorem \ref{thm:domination} will rely on the following result, see Theorem \ref{thm:Z_n} in Section \ref{sec:Z_n}:
For all $a>b$ with $a+b\leq n$, there are $n$-dimensional smooth complex projective varieties whose primitive Hodge numbers $l^{p,q}$ satisfy $l^{a,b}>>0$ and $l^{p,q}=0$ for all other $p>q$.

Theorem \ref{thm:domination} deals with universal inequalities of the form (\ref{eq:domination}). 
In Section \ref{sec:Betti} we deduce from the main results of this paper some progress on the analogous problem for inequalities of arbitrary shape (Corollaries \ref{cor:ineqtruncatedHodgediamond}, \ref{cor:ineq_largedim} and \ref{cor:thm:Z_n}).  
For instance, we will see that any universal inequality among Hodge numbers of smooth complex projective varieties which holds in all sufficiently large dimensions at the same time is a consequence of the Lefschetz conditions. 

The problem of determining all universal inequalities among Hodge numbers of smooth complex projective varieties in a fixed dimension remains open. 
It is however surprisingly easy to solve the analogous problem for inequalities among Betti numbers. 
Indeed, using products of hypersurfaces of high degree, we will prove (Porposition \ref{prop:ineqBettinrs}) that in fact any universal inequality among the Betti numbers of $n$-dimensional smooth complex projective varieties is a consequence of the Lefschetz conditions.

\subsection{Some negative results} \label{intro:negresults}
Theorem \ref{thm:domination} shows that at least in dimension two, the constraints which classical Hodge theory puts on the Hodge numbers of Kähler manifolds are not complete.
Indeed, given weight two Hodge numbers can in general not be realized by a surface --  by Theorem \ref{thm:main2} (resp.\ Theorem \ref{thm:main3}) they can however be realized by higher dimensional varieties. 
In Appendix \ref{sec:3-folds} and \ref{sec:4-folds} of this paper we collect some partial results which demonstrate similar issues in dimensions three and four respectively. 
This is one of the reasons which makes the construction problem for Hodge numbers so delicate. 

In Appendix \ref{sec:3-folds} we prove (Proposition \ref{prop:3-folds}) that the Hodge numbers $h^{p,q}$ of any smooth complex projective three-fold with $h^{1,1}=1$ and $h^{2,0}>0$ satisfy $h^{1,0}=0$, $h^{2,0}<h^{3,0}$, and $h^{2,1}<12^6\cdot h^{3,0}$.
Moreover, for $h^{3,0}-h^{2,0}$ from above bounded, only finitely many deformation types of such examples exist.
In Appendix \ref{sec:4-folds} we prove similar results (Proposition \ref{prop:4-folds}) for projective four-folds with $h^{1,1}=1$.
(The existence of three- and four-folds with $h^{1,1}=1$ and $h^{2,0}>0$ is established by Theorem \ref{thm:main3} in Section \ref{sec:deg2coho}.)

Concerning the Betti numbers, we prove in Appendix \ref{sec:4-folds} (Corollary \ref{cor:4-folds}): 
Let $X$ be a Kähler four-fold with $b_2(X)=1$, then $b_3(X)$ can be bounded in terms of $b_4(X)$.
Since this phenomenon can neither be explained with the Hodge symmetries, the Lefschetz conditions nor the Hodge--Riemann bilinear relations, we conclude that even for the Betti numbers of Kähler manifolds, the known constraints are not complete.

\subsection{Organization of the paper}
In Section \ref{sec:Sommese} we outline our construction methods. 
In Section \ref{sec:hellcurves&groups} we consider the hyperelliptic curve $C_g$ given by $y^2=x^{2g+1}+1$ and construct useful subgroups of $\Aut(C_g^k)$.
In Section \ref{sec:constrmethod} we develop the construction method needed for the proofs of Theorems \ref{thm:main2} and \ref{thm:main1} in Sections \ref{sec:main2} and \ref{sec:main1} respectively.
In Section \ref{sec:deg2coho} we prove Theorem \ref{thm:main3}, i.e.\ we show that for weight two Hodge structures the bound on $h^{1,1}$ in Theorem \ref{thm:main2} can be chosen to be optimal. 
We produce in Section \ref{sec:Z_n} examples whose primitive Hodge numbers $l^{p,q}$ with $p>q$ are concentrated in a single $(p,q)$-type, and show in Section \ref{sec:domination} how our results lead to a proof of Theorem \ref{thm:domination}.
In Section \ref{sec:Betti} we apply our results to the problem of finding universal inequalities among Hodge and Betti numbers of smooth complex projective varieties. 
Finally, we discuss in Appendices \ref{sec:3-folds} and \ref{sec:4-folds} the negative results, mentioned in Section \ref{intro:negresults}.

\subsection{Notation and conventions} \label{subsec:conv}
The natural numbers $\N:=\Z_{\geq 0}$ include zero.
All Kähler manifolds are compact and connected, if not mentioned otherwise.
A variety is a separated integral scheme of finite type over $\C$.
Using the GAGA principle \cite{GAGA}, we usually identify a smooth projective variety with its corresponding analytic space, which is a Kähler manifold.
If not mentioned otherwise, cohomology means singular (or de Rham) cohomology with coefficients in $\C$; the cup product on cohomology will be denoted by $\wedge$.

With a group action $G\times Y\to Y$ on a variety $Y$, we always mean a group action by automorphisms from the left.
For any subgroup $\Gamma\subseteq G$, the fixed point set of the induced $\Gamma$-action on $Y$ will be denoted by 
\begin{align} \label{def:Fix_{Y}(Gamma)}
\Fix_{Y}(\Gamma):=\left\{y\in Y\ |\ g(y)=y\ \text{for all }g \in \Gamma\right\} \ .
\end{align}
If $\Gamma=\left\langle \phi\right\rangle$ is cyclic, then we will frequently write $\Fix_{Y}(\Gamma)=\Fix_{Y}(\phi)$ for this fixed point set.

\section{Outline of our construction methods} \label{sec:Sommese}
The starting point of our constructions is the observation that there are finite group actions $G\times T \to T$, where $T$ is a product of hyperelliptic curves, such that the $G$-invariant cohomology of $T$ is essentially concentrated in a single $(p,q)$-type, see Section \ref{subsec:groupaction}.
In local holomorphic charts, $G$ acts by linear automorphisms.
Thus, by the Chevalley--Shephard--Todd Theorem, $T/G$ is smooth if and only if $G$ is generated by quasi-reflections, that is, by elements whose fixed point set is a divisor on $T$.
Unfortunately, it turns out that in our approach this strong condition can rarely be met.
We therefore face the problem of a possibly highly singular quotient $T/G$.

One way to deal with this problem is to pass to a smooth model $X$ of $T/G$.
However, only the outer Hodge numbers $h^{k,0}$ are birational invariants \cite{K&S}.
Therefore, there will be in general only very little relation between the cohomology of $X$ and the $G$-invariant cohomology of $T$.
Nevertheless, we will find in Section \ref{sec:Z_n} examples $T/G$ which admit smooth models whose cohomology is, apart from (a lot of) additional $(p,p)$-type classes, indeed given by the $G$-invariants of $T$.
We will overcome technical difficulties by a general inductive approach which is inspired by work of Cynk--Hulek \cite{cynk&hulek}, see Proposition \ref{prop:ind.constr}.

In Theorems \ref{thm:main2} and \ref{thm:main1} we need to construct examples with bounded $h^{p,p}$ and so the above method does not work anymore.
Instead, we will use the following lemma, known as the Godeaux--Serre construction, cf. \cite{atiyah-hirzebruch,serre}:

\begin{lemma} \label{lem:Sommese}
Let $G$ be a finite group whose action on a smooth complex projective variety $Y$ is free outside a subset of codimension $>n$.
Then $Y/G$ contains an $n$-dimensional smooth complex projective subvariety whose cohomology below degree $n$ is given by the $G$-invariant classes of $Y$.
\end{lemma}
\begin{proof}
A general $n$-dimensional $G$-invariant complete intersection subvariety $Z\subseteq Y$ is smooth by Bertini's theorem.
For a general choice of $Z$, the $G$-action on $Z$ is free and so $Z/G$ is a smooth subvariety of $Y/G$ which by the Lefschetz hyperplane theorem, applied to $Z\subseteq Y$, has the property we want in the Lemma.
\end{proof}

Roughly speaking, the construction method which we develop in Section \ref{sec:constrmethod} (Proposition \ref{prop:constrmethod}) and which is needed in Theorems \ref{thm:main2} and \ref{thm:main1} works now as follows. 
Instead of a single group action, we will consider a finite number of finite group actions $G_i\times T_i\to T_i$, indexed by $i\in I$. 
Blowing up all $T_i$ simultaneously in a large ambient space $Y$, we are able to construct a smooth complex projective variety $\tilde Y$ which admits an action of the product $G=\prod_{i\in I}G_i$ that is free outside a subset of large codimension and so Lemma \ref{lem:Sommese} applies.
Moreover, the $G$-invariant cohomology of $\tilde Y$ will be given in terms of the $G_i$-invariant cohomology of the $T_i$.
This is a quite powerful method since it allows us to apply Lemma \ref{lem:Sommese} to a finite number of group actions simultaneously -- even without assuming that the group actions we started with are free away from subspaces of large codimension.

\section{Hyperelliptic curves and group actions} \label{sec:hellcurves&groups}

\subsection{Basics on hyperelliptic curves} \label{subsec:hyperell}
In this section, following mostly \cite[pp.\ 214]{shafarevich}, we recall some basic properties of hyperelliptic curves.
In order to unify our discussion, hyperelliptic curves of genus $0$ and $1$ will be $\CP^1$ and elliptic curves, respectively.

For $g\geq 0$, let $f\in \C[x]$ be a degree $2g+1$ polynomial with distinct roots.
Then, a smooth projective model $X$ of the affine curve $Y$ given by
\[
\left\{y^2=f(x)\right\} \subseteq \C^2 
\]
is a hyperelliptic curve of genus $g$.
Although $Y$ is smooth, its projective closure has for $g>1$ a singularity at $\infty$.
The hyperelliptic curve $X$ is therefore explicitly given by the normalization of this projective closure.
It turns out that $X$ is obtained from $Y$ by adding one additional point at $\infty$. 
This additional point is covered by an affine piece, given by
\[
\left\{v^2=u^{2g+2}\cdot f\left(u^{-1}\right)\right\}, \ \text{where}\ x=u^{-1}\ \text{and}\ y= v\cdot u^{-g-1} \ . 
\]
On an appropriate open cover of $X$, local holomorphic coordinates are given by $x,y,u$ and $v$ respectively. 
Moreover, the smooth curve $X$ has genus $g$ and a basis of $H^{1,0}(X)$ is given by the differential forms
\[
\omega_i:=\frac{x^{i-1}}{y} \cdot dx \ ,
\]
where $i=1,\ldots ,g$.

Let us now specialize to the situation where $f$ equals the polynomial $x^{2g+1}+1$ and denote the corresponding hyperelliptic curve of genus $g$ by $C_g$.
It follows from the explicit description of the two affine pieces of $C_g$ that this curve carries an automorphism $\psi_g$ of order $2g+1$ given by
\[
(x,y)\mapsto (\zeta \cdot x,y)\ \ \ \text{and}\ \ \ (u,v)\mapsto (\zeta^{-1}\cdot u,\zeta^{g}\cdot v) \ ,
\]
where $\zeta$ denotes a primitive $(2g+1)$-th root of unity.
Similarly, 
\[
(x,y)\mapsto (x,-y)\ \ \ \text{and}\ \ \ (u,v)\mapsto (u,- v) \ ,
\]
defines an involution which we denote by multiplication with $-1$.
Moreover, it follows from the above description of $H^{1,0}(C_g)$ that the $\psi_g$-action on $H^{1,0}(C_g)$ has eigenvalues $\zeta,\ldots ,\zeta^g$, whereas the involution acts by multiplication with $-1$ on $H^{1,0}(C_g)$.

Any smooth curve can be embedded into $\CP^3$.
For the curve $C_g$, we fix the explicit embedding which is given by 
\[
[1:x:y:x^{g+1}]=[u^{g+1}:u^{g}:v:1] \ .
\]
Obviously, the involution as well as the order $(2g+1)$-automorphism $\psi_g$ of $C_g\subseteq \CP^3$ extend to $\CP^3$ via
\[
[1:1:-1:1]\ \ \ \text{and}\ \ \ [1:\zeta:1:\zeta^{g+1}]
\]
respectively.

\subsection{Group actions on products of hyperelliptic curves} \label{subsec:groupaction}
Let 
\[
T:=C_g^k
\]
be the $k$-fold product of the hyperelliptic curve $C_g$ with automorphism $\psi_g$ defined in Section \ref{subsec:hyperell}.
For $a\geq b$ with $a+b=k$, we define for each $i=1,2,3$ a subgroup $G^i(a,b,g)$ of $\Aut(T)$ whose elements are called automorphisms of the $i$-th kind. 
The subgroup of automorphisms of the first kind is given by
\[
G^1(a,b,g):=\left\{\psi_g^{j_1}\times \cdots \times \psi_g^{j_{a+b}} \ |\  j_1+\cdots +j_a-j_{a+1}-\cdots - j_{a+b} \equiv 0\mod (2g+1)\right\} \ .
\]

In order to define the automorphisms of the second kind, let us consider the group $ \Sym(a)\times \Sym(b)\times \mu_2^{a+b}$, where $\mu_2=\left\{1,-1\right\}$ is the multiplicative group on two elements.
An element $(\sigma,\tau,\epsilon)$, where $\sigma\in \Sym(a)$, $\tau\in \Sym(b)$ and $\epsilon=(\epsilon_1,\ldots, \epsilon_{a+b})$ is a vector of signs $\epsilon_i\in \left\{1,-1\right\}$, acts on $T$ via
\[
\left(x_1,\ldots , x_{a},y_{1},\ldots ,y_{b}\right)\mapsto \left(\epsilon_1\cdot x_{\sigma(1)},\ldots ,\epsilon_a \cdot x_{\sigma(a)},\epsilon_{a+1} \cdot y_{\tau(1)},\ldots ,\epsilon_{a+b}\cdot y_{\tau(b)}\right) \ .
\]
Here, multiplication with $-1$ means that we apply the involution $(-1)\in \Aut(C_g)$. 
We define 
\[
G^2(a,b,g) \subseteq \Sym(a)\times \Sym(b)\times \mu_2^{a+b}
\]
to be the index four subgroup consisting of those elements $(\sigma,\tau,\epsilon)$ which satisfy
\[
\sign(\sigma)\cdot \epsilon_1\cdot \ldots \cdot \epsilon_{a}=1 \ \ \text{and}\ \ \sign(\tau)\cdot \epsilon_{a+1}\cdot \ldots \cdot \epsilon_{a+b}=1 ,
\]
where $\sign$ denotes the signum of the corresponding permutation.
Via the above action of $\Sym(a)\times \Sym(b)\times \mu_2^{a+b}$ on $T$, the group $G^2(a,b,g)$ is a finite subgroup of $\Aut(T)$.

Finally, $G^3(a,b,g)$ is trivial, if $a\neq b$ and if $a=b$, then it is generated by the automorphism which interchanges the two factors of $T={C_g}^a\times {C_g}^a$.

\begin{definition} \label{def:G}
The group $G(a,b,g)$ is the subgroup of $\Aut(T)$ which is generated by the union of $G^i(a,b,g)$ for $i=1,2,3$.
\end{definition}

Automorphisms of different kinds do in general not commute with each other.
However, it is easy to see that each element in $G(a,b,g)$ can be written as a product $\phi_1\circ \phi_2\circ \phi_3$ such that $\phi_i$ lies in $G^i(a,b,g)$. 
Therefore, $G(a,b,g)$ is a finite group which naturally acts on the cohomology of $T$.

\begin{lemma} \label{lem:groupaction}
If $a>b$, then the $G(a,b,g)$-invariant cohomology of $T$ is a direct sum 
\[
V^{a,b}\oplus V^{b,a}\oplus \left(\bigoplus_{p=0}^k V^{p,p}\right) \ ,
\] 
where $V^{a,b}=\overline{V^{b,a}}$ is a $g$-dimensional space of $(a,b)$-classes and $V^{p,p}\cong V^{k-p,k-p}$ is a space of $(p,p)$-classes of dimension $\min(p+1,b+1)$, where $p\leq k/2$ is assumed.
\end{lemma}

\begin{proof}
We denote the fundamental class of the $j$-th factor of $T$ by $\Omega_j\in H^{1,1}(T)$.
Moreover, we pick for $j=1,\ldots ,k$ a basis $\omega_{j1},\ldots ,\omega_{jg}$ of $(1,0)$-classes of the $j$-th factor of $T$ in such a way that 
\[
\psi_g^\ast\omega_{jl}=\zeta^{l}\omega_{jl}
\]
for a fixed $(2g+1)$-th root of unity $\zeta$. 
Then the cohomology ring of $T$ is generated by the $\Omega_j$'s, $\omega_{jl}$'s and their conjugates.
Moreover, the involution on the $j$-th curve factor of $T$ acts on $\omega_{jl}$ and $\overline{\omega_{jl}}$ by multiplication with $-1$ and leaves $\Omega_j$ invariant.

Suppose that we are given a $G(a,b,g)$-invariant class which contains the monomial
\begin{align} \label{eq:monomial}
\Omega_{i_1}\wedge \cdots \wedge\Omega_{i_s}\wedge \omega_{j_1l_1}\wedge \cdots\wedge \omega_{j_rl_r}\wedge \overline{\omega_{j_{r+1}\:l_{r+1}}}\wedge \cdots \wedge \overline{\omega_{j_t\:l_{t}}}
\end{align}
nontrivially.
Since the product of a $(1,0)$- and a $(0,1)$-class of the $i$-th curve factor is a multiple of $\Omega_i$, and since classes of degree $3$ vanish on curves, we may assume that the indices $i_1,\ldots ,i_s,j_1,\ldots ,j_t$ are pairwise distinct.
Therefore, application of a suitable automorphism of the first kind shows $t=0$ if $s\geq 1$ and $t=a+b$ if $s=0$.
In the latter case, suppose that there are indices $i_1$ and $i_2$ with either $i_1,i_2\leq r$ or $i_1,i_2>r$, such that  $j_{i_1}\leq a$ and $j_{i_2}> a$ holds.
Then, application of a suitable automorphism of the first kind yields $l_{i_1}+l_{i_2}=0$ in $\Z/(2g+1)\Z$, which contradicts $1\leq l_i \leq g$.
This shows
\[
\left\{j_1,\ldots ,j_r\right\}=\left\{1,\ldots ,a\right\}\ \ \ \text{or}\ \ \ \left\{j_1,\ldots ,j_r\right\}=\left\{a+1,\ldots ,a+b\right\} \ .
\]
By applying suitable automorphisms of the first kind once more, one obtains $l_1=\cdots =l_t$. 
Thus, we have just shown that a $G(a,b,g)$-invariant class of $T$ is either a polynomial in the $\Omega_j$'s, or a linear combination of
\begin{align} \label{eq:omega_l}
\omega_l:=\omega_{1l}\wedge \cdots \wedge \omega_{al}\wedge \overline{\omega_{a+1\: l}}\wedge \cdots \wedge \overline{\omega_{a+b\: l}} ,
\end{align}
or their conjugates, where $l=1,\ldots ,g$.
Note that $\omega_l$ is of $(a,b)$-type whereas any polynomial in the $\Omega_j$'s is a sum of $(p,p)$-type classes.
Moreover, by the definition of $G^1(a,b,g)$ and $G^2(a,b,g)$, both groups act trivially on $\omega_l$ and $\overline \omega_l$.
Since $a>b$, the group $G^3(a,b,g)$ is trivial and so it follows that each $\omega_l$ and $\overline{\omega}_l$ is $G(a,b,g)$-invariant.
Therefore, the span of $\omega_1,\ldots ,\omega_g$ yields a $g$-dimensional space $V^{a,b}$ of $G(a,b,g)$-invariant $(a,b)$-classes.
Its conjugate $V^{b,a}:=\overline{V^{a,b}}$ is spanned by the $G(a,b,g)$-invariant $(b,a)$-classes $\overline{\omega}_1,\ldots ,\overline{\omega}_g$.

Next, we define $V^{p,p}$ to consist of all $G(a,b,g)$-invariant homogeneous degree $p$ polynomials in $\Omega_1,\ldots ,\Omega_{a+b}$.
Application of a suitable automorphism of the second kind shows that any element $\Theta$ in $V^{p,p}$ is a polynomial in the elementary symmetric polynomials in $\Omega_1,\ldots ,\Omega_a$ and $\Omega_{a+1},\ldots ,\Omega_{a+b}$.
By standard facts about symmetric polynomials, it follows that $\Theta$ can be written as a polynomial in 
\[
\sum_{j=1}^a{\Omega_j}^i \ \ \ \text{and}\ \ \  \sum_{j=a+1}^{a+b}{\Omega_j}^i
\]
for $i\geq 0$.
Since $\Omega_j^2$ vanishes for all $j$, we see that a basis of $V^{p,p}$ is given by the elements
\[
\left(\Omega_1+\cdots +\Omega_a\right)^{p-i}\wedge\left(\Omega_{a+1}+\cdots +\Omega_{a+b}\right)^{i} \ ,
\]
where $0\leq p-i \leq a$ and $0\leq i\leq b$.
Using $a> b$, this concludes the Lemma by an easy counting argument.
\end{proof}

\begin{lemma} \label{lem:groupactiona=b}
If $a=b$, then the $G(a,b,g)$-invariant cohomology of $T$ is a direct sum 
$
\bigoplus_{p=0}^k V^{p,p}
$, 
where $V^{p,p}\cong V^{k-p,k-p}$ is a space of $(p,p)$-classes whose dimension is given by $\left\lfloor p/2\right\rfloor+1$, if $p<a$, and by $\left\lfloor p/2\right\rfloor +g+1$, if $p=a$.
\end{lemma}

\begin{proof}
We use the same notation as in the proof of Lemma \ref{lem:groupaction} and put $b:=a$.
Suppose that we are given a $G(a,a,g)$-invariant cohomology class on $T$ which contains the monomial (\ref{eq:monomial}) nontrivial.
This monomial is then necessarily $G^1(a,a,g)$-invariant and the same arguments as in Lemma \ref{lem:groupaction} show that it is either a monomial in the $\Omega_j$'s, or it coincides with one of the $\omega_l$'s and their conjugates, defined in (\ref{eq:omega_l}).

For each $l=1,\ldots ,g$, the classes $\omega_l$ and $\overline {\omega_l}$ are invariant under the action of $G^1(a,a,g)$ and $G^2(a,a,g)$.
Moreover, the generator of $G^3(a,a,g)$ interchanges the two factors of $T={C_g}^a\times {C_g}^a$.
Its action on cohomology therefore maps $\omega_l$ to $(-1)^a\cdot \overline{\omega_l}$.
This shows that a linear combination of the $\omega_l$'s and their conjugates is $G(a,a,g)$-invariant if and only if it is a linear combination of the classes
\begin{align} \label{eq:w_l+overlinew_l}
\omega_l+(-1)^a\cdot \overline{\omega_l} \ ,
\end{align}
where $l=1,\ldots ,g$.
This yields a $g$-dimensional space of $G(a,a,g)$-invariant $(a,a)$-classes.

It remains to study which homogeneous polynomials in the $\Omega_j$'s are $G(a,a,g)$-invariant.
As in the proof of Lemma \ref{lem:groupaction}, one shows that any such polynomial of degree $p$ is necessarily a linear combination of
\[
\Omega(p-i,i):=\left(\Omega_1+\cdots +\Omega_a\right)^{p-i}\wedge\left(\Omega_{a+1}+\cdots +\Omega_{2a}\right)^{i} \ ,
\]
where $0\leq p-i \leq a$ and $0\leq i\leq a$.
The above monomials are clearly invariant under the action of $G^1(a,a,g)$ and $G^2(a,a,g)$.
Moreover, the generator of $G^3(a,a,g)$ interchanges the two factors of $T$ and hence its action on cohomology maps $\Omega(p-i,i)$ to $\Omega(i,p-i)$.
We are therefore reduced to linear combinations of 
\[
\Omega(i,p-i)+\Omega(p-i,i) \ ,
\]
where $0\leq i\leq p-i\leq a$.
Such linear combinations are certainly $G(a,a,g)$-invariant.
If $p\leq a$, then the condition on the index $i$ means $0\leq i\leq p/2$.
It follows that for $p\leq a$, the space of those $G(a,a,g)$-invariant $(p,p)$-classes which are given by polynomials in the $\Omega_j$'s has dimension $\left\lfloor p/2\right\rfloor+1$.
Combining this with our previous observation that the classes in (\ref{eq:w_l+overlinew_l}) span a $g$-dimensional space of $G(a,a,g)$-invariant $(a,a)$-classes, this concludes the Lemma.
\end{proof}

For later applications, we will also need the following:

\begin{lemma} \label{lem:extendedaction}
For all $a\geq b$ there exists some $N>0$ and an embedding of $G(a,b,g)$ into $\GL(N+1)$ such that a $G(a,b,g)$-equivariant embedding of ${C_g}^{a+b}$ into $\CP^N$ exists.
Moreover, ${C_g}^{a+b}$ contains a point which is fixed by $G(a,b,g)$.
\end{lemma}

\begin{proof}
For the first statement, we use the embedding of $C_g$ into $\CP^3$, constructed in Section \ref{subsec:hyperell}.
This yields an embedding of ${C_g}^{a+b}$ into $(\CP^3)^{a+b}$.
From the explicit description of that embedding, it follows that the action of $G(a,b,g)$ on ${C_g}^{a+b}$  extends to an action on $(\CP^3)^{a+b}$ which is given by first multiplying homogeneous coordinates with some roots of unity and then permuting these in some way.
Using the Segre map, we obtain for some large $N$ an embedding of $G(a,b,g)$ into $\GL(N+1)$ together with a $G(a,b,g)$-equivariant embedding
\[
{C_g}^{a+b}\hookrightarrow \CP^{N} \ .
\] 
This proves the first statement in the Lemma.

For the second statement, note that the point $\infty$ of $C_g$ is fixed by both, $\psi_g$ as well as the involution.
Thus, $\infty$ yields a point on the diagonal of ${C_g}^{a+b}$ which is fixed by $G(a,b,g)$.
\end{proof}

\section{Group actions on blown-up spaces}  \label{sec:constrmethod}
\subsection{Cohomology of blow-ups} \label{subsec:blow-up}
Let $Y$ be a Kähler manifold, $T$ a submanifold of codimension $r$ and let $\pi: \tilde Y \rightarrow Y$ be the blow-up of $Y$ along $T$.
Then the exceptional divisor $j:E\hookrightarrow \tilde Y$ of this blow-up is a projective bundle of rank $r-1$ over $T$ and we denote the dual of the tautological line bundle on $E$ by $\mathcal O_E(1)$.
Then the Hodge structure on $\tilde Y$ is given by the following theorem, see \cite[p.\ 180]{voisin1}.

\begin{theorem} \label{thm:blow-up}
We have an isomorphism of Hodge structures
\[
H^k(Y,\Z)\oplus \left(\bigoplus _{i=0}^{r-2} H^{k-2i-2}(T,\Z)\right)\rightarrow H^k\left(\tilde Y,\Z\right) \ ,
\]
where on $H^{k-2i-2}(T,\Z)$, the natural Hodge structure is shifted by $(i+1,i+1)$.
On $H^k(Y,\Z)$, the above morphism is given by $\pi^\ast$ whereas on $H^{k-2i-2}(T,\Z)$ it is given by $j_\ast\circ h^{i}\circ\pi|_E^\ast$, where $h$ denotes the cup product with $c_1(\mathcal O_E(1))\in H^2(E,\Z)$ and $j_\ast$ is the Gysin morphism of the inclusion $j:E\hookrightarrow \tilde Y$.
\end{theorem}

We will need the following property of the ring structure of $H^\ast(\tilde Y,\Z)$.
Note that the first Chern class of $\mathcal O_E(1)$ coincides with the pullback of $-[E]\in H^2(\tilde Y,\Z)$ to $E$.
For a class $\alpha \in H^{k-2i-2}(T,\Z)$, this implies:
\begin{align} \label{eq:ringstr:blow-up}
(j_\ast\circ h^{i}\circ\pi|_E^\ast) (\alpha)=j_\ast (j^\ast (-[E])^i\wedge \pi|_E^\ast(\alpha)) =(-[E])^i\wedge j_\ast(\pi|_E^\ast (\alpha)) \ ,
\end{align}
where we used the projection formula.

\subsection{Key construction} \label{subsec:keyconstr}
Let $I$ be a finite nonempty set, and let $i_0\in I$.
Suppose that for each $i\in I$, we are given a representation
\[
G_i\to GL(V_i) 
\] 
of a finite group $G_i$ on a finite dimensional complex vector space $V_i$. 
Further, assume that the induced $G_i$-action on $\CP(V_i)$ restricts to an action on a smooth subvariety $T_i\subseteq \CP(V_i)$ and that there is a point $p_{i_0}\in T_{i_0}$ which is fixed by $G_{i_0}$.
Then we have the following key result.

\begin{proposition} \label{prop:constrmethod}
For any $n>0$, there exists some complex vector space $V$ and pairwise disjoint embeddings of $T_i$ into  $Y:=T_{i_0}\times \mathbb P(V)$, such that the blow-up $\tilde Y$ of $Y$ along all $T_i$ with $i\neq i_0$ inherits an action of $G:=\prod_{i\in I} G_i$ which is free outside a subset of codimension $>n$.
Moreover, $\tilde Y/G$ contains an $n$-dimensional smooth complex projective subvariety $X$ whose primitive Hodge numbers are, for all $p+q<n$, given by
\[
l^{p,q}(X)=\dim\left(H^{p,q}(T_{i_0})^{G_{i_0}}\right)+\sum_{i\neq i_0} \dim \left(H^{p-1,q-1}(T_i)^{G_i}\right) \ .
\]
\end{proposition}

\begin{proof}
The product 
\[
G:=\prod_{i\in I}G_i
\]
acts naturally on the direct sum $\bigoplus_{i\in I}V_i$.
We pick some $k>>0$.
Then
\[
V:=(\bigoplus_{i\in I}V_i)\oplus  (\bigoplus_{g\in G}g\cdot \C^k)
\]
inherits a linear $G$-action where $h\in G$ acts on the second factor by sending $g\cdot \C^k$ canonically to $(h\cdot g) \cdot \C^k$. 
Then we obtain $G$-equivariant inclusions
\[
T_i\hookrightarrow \mathbb P(V_i) \hookrightarrow \mathbb P(V) \ ,
\]
where for $j\neq i$, the group $G_j$ acts via the identity on $T_i$ and $\mathbb P(V_i)$.
The product
\[
Y:=T_{i_0}\times \mathbb P(V) 
\]
inherits a $G$-action via the diagonal, where for $i\neq i_0$ elements of $G_i$ act trivially on $T_{i_0}$.

Using the base point $p_{i_0}\in T_{i_0}$, we obtain for all $i\in I$ disjoint inclusions
\[
T_i\hookrightarrow Y \ ,
\]
and we denote the blow-up of $Y$ along the union of all $T_i$ with $i\neq i_0$ by $\tilde Y$. 
Since $p_{i_0}\in T_{i_0}$ is fixed by $G$, the $G$-action maps each $T_i$ to itself and hence lifts to $\tilde Y$.

We want to prove that the $G$-action on $\tilde Y$ is free outside a subset of codimension $>n$.
For $k$ large enough, the $G$-action on $Y$ certainly has this property.
Hence, it suffices to check that the induced $G$-action on the exceptional divisor $E_j$ above $T_j\subseteq Y$ is free outside a subset of codimension $>n$.

For $|I|=1$, this condition is empty.
For $|I|\geq2$, we fix an index $j\in I$ with $j\neq i_0$.
Then it suffices to show that for a given nontrivial element $\phi\in G$ the fixed point set $\Fix_{E_j}(\phi)$ has codimension $>n$ in $E_j$.
If $t_j\in T_j$ is not fixed by $\phi$, then the fiber of $E_j\to T_j$ above $t_j$ is moved by $\phi$ and hence disjoint from $\Fix_{E_j}(\phi)$.
Conversely, if $t_j$ is fixed by $\phi$, then $\phi$ acts on the normal space 
\[
\mathcal N_{T_j,t_j}=T_{Y,t_j}/T_{T_j,t_j} 
\]
via a linear automorphism and the projectivization of this vector space is the fiber of $E_j\to T_j$ above $t_j$.
The tangent space $T_{Y,t_j}$ equals
\[
T_{T_{i_0},p_{i_0}}\oplus \left(L^\ast\otimes \left(V/L\right)\right) \ ,
\]
where $L$ is the line in $V$ which corresponds to the image of $t_j$ under the projection $Y\to \mathbb P(V)$.
Since $\phi\neq \id$, it follows for large $k$ that the fixed point set of $\phi$ on the fiber of $E_j$ above $t_j$ has codimension $>n$.
Hence, $\Fix_{E_j}(\phi)$ has codimension $>n$ in $E_j$, as we want. 

As we have just shown, the $G$-action on $\tilde Y$ is free outside a subset of codimension $>n$.
Hence, by Lemma \ref{lem:Sommese}, the quotient $\tilde Y/G$ contains an $n$-dimensional smooth complex projective subvariety $X$ whose cohomology below the middle degree is given by the $G$-invariants of $\tilde Y$.
In order to calculate the dimension of the latter, we first note that for all $i\in I$, the divisor $E_i$ on $\tilde Y$ is preserved by $G$. 
Since $\mathcal O_{E_i}(-1)$ is given by the restriction of $\mathcal O_{\tilde Y}(E_i)$ to $E_i$, it follows that $c_1(\mathcal O_{E_i}(1))$ is $G$-invariant.
For $p+q<n$, the primitive $(p,q)$-th Hodge number of $X$ is by Theorem \ref{thm:blow-up} therefore given by:
\[
l^{p,q}(X)= \dim(H^{p,q}(Y)^G)-\dim(H^{p-1,q-1}(Y)^G)+\sum_{i\neq i_0} \dim \left(H^{p-1,q-1}(T_i)^{G_i}\right) \ ,
\]
where $H^\ast(-)^G$ denotes $G$-invariant cohomology.
Since any automorphism of projective space acts trivially on its cohomology, the Künneth Formula implies 
\[
 \dim(H^{p,q}(Y)^G)-\dim(H^{p-1,q-1}(Y)^G)=\dim\left(H^{p,q}(T_{i_0})^{G_{i_0}}\right) \ .
\]
This finishes the proof of Proposition \ref{prop:constrmethod}.
\end{proof}

\section{Proof of Theorem \ref{thm:main2}} \label{sec:main2}
Fix $k\geq 1$ and let $(h^{p,q})_{p+q=k}$ be a symmetric collection of natural numbers. 
In the case where $k=2m$ is even, we additionally assume
\[
h^{m,m}\geq m\cdot \left(m-\left\lfloor \frac{m}{2}\right\rfloor+1\right)+\left\lfloor \frac{m}{2}\right\rfloor^2 \ .
\]
Then we want to construct for $n>k$ an $n$-dimensional smooth complex projective variety $X$ with the above Hodge numbers on $H^{k}(X,\C)$.

Let us consider the index set $I:=\left\{0,\ldots ,\left\lfloor (k-1)/2\right\rfloor\right\}$ and put $i_0:=0$.
Then, for all $i\in I$, we consider the $(k-2i)$-fold product
\[
T_i:=\left(C_{h^{k-i,i}}\right)^{k-2i} \ ,
\]
where $C_{h^{k-i,i}}$ denotes the hyperelliptic curve of genus $h^{k-i,i}$, defined in Section \ref{subsec:hyperell}. 
On $T_i$ we consider the action of
\[
G_i:=G(k-2i\:,0\:,\:h^{k-i,i}) \ , 
\]
defined in Section \ref{subsec:groupaction}.

By Lemma \ref{lem:extendedaction}, we may apply the construction method of Section \ref{subsec:keyconstr} to the set of data $(T_i,G_i,I,i_0)$.
Thus, by Proposition \ref{prop:constrmethod}, there exists an $n$-dimensional smooth complex projective variety $X$ whose primitive Hodge numbers are for $p+q<n$ given by
\[
l^{p,q}(X)=\dim\left(H^{p,q}(T_{i_0})^{G_{i_0}}\right)+\sum_{i\neq i_0} \dim \left(H^{p-1,q-1}(T_i)^{G_i}\right) \ .
\]
Lemma \ref{lem:groupaction} says that for $p>q$, the only $G_i$-invariant $(p,q)$-classes on $T_i$ are of type $(k-2i,0)$. 
Therefore, $l^{p,q}(X)$ vanishes for $p>q$ and $p+q<n$ in all but the following cases:
\[
l^{k,0}(X)=\dim\left(H^{k,0}(T_{i_0})^{G_{i_0}}\right)=h^{k,0} \ ,
\]
and
\[
l^{k-2i+1,1}(X)=\dim\left(H^{k-2i,0}(T_{i})^{G_{i}}\right)=h^{k-i,i} \ ,
\]
for all $1\leq i<k/2$.
Using the formula 
\[
h^{k-i,i}(X)=\sum_{s=0}^i l^{k-i-s,i-s}(X) \ ,
\]
we deduce for $0\leq i<k/2$:
\[
h^{k-i,i}(X)=h^{k-i,i} \ .
\]
Thus, if $k$ is odd, then the Hodge symmetries imply that the Hodge structure on $H^k(X,\C)$ has Hodge numbers $(h^{k,0},\ldots ,h^{0,k})$.

We are left with the case where $k=2m$ is even.
Since blowing-up a point increases $h^{m,m}$ by one and leaves $h^{p,q}$ with $p\neq q$ unchanged, it suffices to prove
\[
h^{m,m}(X)=m\cdot \left(m-\left\lfloor \frac{m}{2}\right\rfloor+1\right)+\left\lfloor \frac{m}{2}\right\rfloor^2 \ .
\]
As we have seen:
\[
h^{m,m}(X)=\sum_{s=0}^{m}l^{s,s}(X)=\sum_{s=0}^{m}\left(\dim \left(H^{s,s}(T_0)^{G_0}\right)+\sum_{0<i<k/2}\dim \left(H^{s-1,s-1}(T_i)^{G_i}\right)\right) \ .
\]
By Lemma \ref{lem:groupaction}, we have $\dim \left(H^{s,s}(T_i)^{G_i}\right)=1$ for all $0\leq s \leq 2\cdot \dim (T_i)$ and so
\[
h^{m,m}(X)=m+1+\sum_{s=0}^{m-1}\sum_{0<i<k/2}\dim \left(H^{s,s}(T_i)^{G_i}\right) \ .
\]
Since $T_i$ has dimension $2(m-i)$, we see that
\[
\sum_{s=0}^{m-1}\dim \left(H^{s,s}(T_i)^{G_i}\right)=
\begin{cases}
  m\ ,  & \text{if }2(m-i)> m-1\ ,\\
  2(m-i)+1 \ , & \text{if }2(m-i)\leq m-1\ .
\end{cases}
\]
Hence
\[
h^{m,m}(X)	=m+1 +\sum_{i=1}^{\left\lfloor m/2\right\rfloor} m+\sum_{i=\left\lfloor m/2\right\rfloor+1}^{m-1}(2(m-i)+1)\ ,
\]
and it is straightforward to check that this simplifies to
\begin{align*}
h^{m,m}(X)		=m\cdot \left\lfloor (m+3)/2\right\rfloor+\left\lfloor m/2\right\rfloor^2 \ .
\end{align*}
This finishes the proof of Theorem \ref{thm:main2}.

In Theorem \ref{thm:main2} we have only dealt with Hodge structures below the middle degree.
Under stronger assumptions, the following corollary of Theorem \ref{thm:main2} deals with Hodge structures in the middle degree.
We will use this corollary in the proof of Theorem \ref{thm:domination} in Section \ref{sec:domination}.

\begin{corollary} \label{cor:main2}
Let $(h^{n,0},\ldots ,h^{0,n})$ be a symmetric collection of even natural numbers such that $h^{n,0}=0$.
If $n=2m$ is even, then we additionally assume
	\[
	h^{m,m}\geq 2\cdot (m-1)\cdot \left\lfloor (m+2)/2\right\rfloor+2\cdot \left\lfloor (m-1)/2\right\rfloor ^2 \ .
	\]
Then there exists an $n$-dimensional smooth complex projective variety $X$ whose Hodge structure of weight $n$ realizes the given Hodge numbers.
\end{corollary}

\begin{proof}
For $n=1$ we may put $X=\CP^1$ and for $n=2$ the blow-up of $\CP^2$ in $h^{1,1}-1$ points does the job.
It remains to deal with $n\geq 3$. 
Here, by Theorem \ref{thm:main2} there exists an $(n-1)$-dimensional smooth complex projective variety $Y$ whose Hodge decomposition on $H^{n-2}(Y,\C)$ has Hodge numbers 
\[
(\frac{1}{2}\cdot h^{n-1,1},\ldots , \frac{1}{2}\cdot h^{1,n-1}) \ .
\]
By the Künneth Formula, the product 
$
X:=Y\times \CP^1
$ 
has Hodge numbers
\[
h^{p,q}(X)=h^{p,q}(Y)+h^{p-1,q-1}(Y) \ .
\]
Using the Hodge symmetries on $Y$, Corollary \ref{cor:main2} follows.
\end{proof}

\section{Proof of Theorem \ref{thm:main1}} \label{sec:main1}
In this section we prove Theorem \ref{thm:main1}, stated in the Introduction.
Our proof will follow the same lines as the proof of Theorem \ref{thm:main2} in Section \ref{sec:main2}.

Given a truncated $n$-dimensional formal Hodge diamond whose Hodge numbers (resp.\ primitive Hodge numbers) are denoted by $h^{p,q}$ (resp.\ $l^{p,q}$).
Suppose that one of the following two additional conditions holds:
\begin{enumerate}
	\item The number $h^{k,0}$ vanishes for all $k\neq k_0$ for some $k_0\in\left\{1,\ldots ,n-1\right\}$. \label{item1:proof:main1}
	\item The number $h^{k,0}$ vanishes for all $k=1,\ldots , n-3$. \label{item2:proof:main1}
\end{enumerate}
We will construct universal constants $C(p,n)$ such that under the additional assumption $l^{p,p}\geq C(p,n)$ for all $1\leq p < n/2$, an $n$-dimensional smooth complex projective variety $X$ with the given truncated Hodge diamond exists.
Then Theorem \ref{thm:main1} follows as soon as we have shown $C(p,n)\leq p\cdot (n^2-2n+5)/4 $.

Since blowing-up a point on $X$ increases the primitive Hodge number $l^{1,1}(X)$ by one and leaves the remaining primitive Hodge numbers unchanged, it suffices to deal with the case where $l^{1,1}=C(1,n)$ is minimal.

To explain our construction, let us for each $r\geq s>0$ with $2<r+s<n$ consider the $(r+s-2)$-fold product
\[
T_{r,s}:=\left({C_{l^{r,s}}}\right)^{r+s-2} \ ,
\]
where $C_{l^{r,s}}$ is the hyperelliptic curve of genus $l^{r,s}$, constructed in Section \ref{subsec:hyperell}.
On $T_{r,s}$ we consider the group action of
\[
G_{r,s}:=G(r-1,s-1,l^{r,s}) \ ,
\]
defined in Section \ref{subsec:groupaction}.

At this point we need to distinguish between the above cases (\ref{item1:proof:main1}) and (\ref{item2:proof:main1}).
We begin with (\ref{item1:proof:main1}) and consider the index set
\[
I:=\left\{(r,s):\ r\geq s >0,\ \ n>r+s>2\right\}\cup\left\{i_0\right\} \ ,
\]
and put
\[
T_{i_0}:=\left({C_{l^{k_0,0}}}\right)^{k_0}\ \ \ \text{and}\ \ \ G_{i_0}:=G(k_0,0,l^{k_0,0}) \ .
\]
By Lemma \ref{lem:extendedaction}, we may apply the construction method of Section \ref{subsec:keyconstr} to the set of data $(T_i,G_i,I,i_0)$.
Thus, Proposition \ref{prop:constrmethod} yields an $n$-dimensional smooth complex projective variety $X$ whose primitive Hodge numbers $l^{p,q}(X)$ with $p+q<n$ are given by
\begin{equation} \label{eq1:main1}
l^{p,q}(X)=\dim\left(H^{p,q}(T_{i_0})^{G_{i_0}}\right)+\sum_{(r,s)\in I\setminus \left\{i_0\right\}}\dim\left(H^{p-1,q-1}(T_{r,s})^{G_{r,s}}\right) \ .
\end{equation}
If $p>q$, then Lemmas \ref{lem:groupaction} and \ref{lem:groupactiona=b} say that 
\begin{align} \label{eq:h^{p-1,q-1}(T_{r,s})}
\dim \left(H^{p-1,q-1}(T_{r,s})^{G_{r,s}}\right) =
	\begin{cases}
	0		&\text{if $(r,s)\neq (p,q)$} \ , \\
 	l^{p,q} &\text{if $(r,s)=(p,q)$} \ .
 	\end{cases}
\end{align}
Moreover,
\begin{align} \label{eq:h^{p-1,q-1}(T_{i_0})}
\dim \left(H^{p,q}(T_{i_0})^{G_{i_0}}\right) =
	\begin{cases}
	0		&\text{if $(k_0,0)\neq (p,q)$} \ , \\
 	l^{p,q} &\text{if $(k_0,0)=(p,q)$} \ .
 	\end{cases}
\end{align}
In (\ref{eq1:main1}), the summation condition $(r,s)\in I\setminus \left\{i_0\right\}$ means $r\geq s>0$ and $n>r+s>2$.
It therefore follows from (\ref{eq:h^{p-1,q-1}(T_{r,s})}) and (\ref{eq:h^{p-1,q-1}(T_{i_0})}) that $l^{p,q}(X)=l^{p,q}$ holds for all $p> q$ with $p+q<n$.
By the Hodge symmetries on $X$, $l^{p,q}(X)=l^{p,q}$ then follows for all $p\neq q$ with $p+q<n$.

Next, for $p=q$, one extracts from (\ref{eq1:main1}) an explicit formula of the form 
\[
l^{p,p}(X)=l^{p,p}+C_1(p,n) \ ,
\]
where $C_1(p,n)$ is a constant which only depends on $p$ and $n$.
Replacing $l^{p,p}$ by $l^{p,p}-C_1(p,n)$ in the above argument then shows that in case (\ref{item1:proof:main1}), an $n$-dimensional smooth complex projective variety with the given truncated Hodge diamond exists as long as 
\[
l^{p,p}\geq C_1(p,n)
\] 
holds for all $1\leq p < n/2$.

In order to find a rough estimation for $C_1(p,n)$, we deduce from Lemmas \ref{lem:groupaction} and \ref{lem:groupactiona=b} the following inequalities
\[
\dim \left(H^{p,p}(T_{i_0})^{G_{i_0}}\right)\leq1 \ \ \ \text{for all $p$} \ ,
\]
and 
\[
\dim \left(H^{p-1,p-1}(T_{r,s})^{G_{r,s}}\right) \leq
	\begin{cases}
	p		&\text{if $(r,s)\neq (p,p)$} \ , \\
 	p+l^{p,p} &\text{if $(r,s)=(p,p)$} \ .
 	\end{cases}
\]
Using these estimates, (\ref{eq1:main1}) gives
\begin{align} \label{eq:C_1}
C_1(p,n)\leq 1+\sum_{\substack{r\geq s>0 \\ n>r+s>2}} p \ ,
\end{align}
where we used that $(r,s)\in I\setminus\left\{i_0\right\}$ is equivalent to $r\geq s>0$ and $n>r+s>2$.
If we write $\left\lfloor x\right\rfloor$ for the floor function of $x$, then (\ref{eq:C_1}) gives explicitly:
\[
C_1(p,n)\leq p\cdot n\cdot \left\lfloor \frac{n-1}{2}\right\rfloor -p\cdot \left\lfloor \frac{n-1}{2} \right\rfloor \cdot \left( \left\lfloor  \frac{n-1}{2} \right\rfloor +1\right)  \ .
\]
If $n$ is odd, then the above right-hand-side equals $p\cdot(n-1)^2/4$ and if $n$ is even, then it is given by $p\cdot n(n-2)/4$.
Hence,
\[
C_1(p,n)\leq p\cdot (n-1)^2/4 \ .
\]

Let us now turn to case (\ref{item2:proof:main1}).
Here we  consider the same index set $I$ as above, and for all $i\neq i_0$ we also define $T_i$ and $G_i$ as above. 
However, for $i=i_0$, we put
\[
T_{i_0}:=\left(C_{l^{n-1,0}}\right)^{n-1}\times\left(C_{l^{n-2,0}}\right)^{n-2}
\]
and
\[
G_{i_0}:=G(n-1,0,l^{n-1,0})\times G(n-2,0,l^{n-2,0}) \ .
\]
By Lemma \ref{lem:extendedaction}, there exist integers $N_1$ and $N_2$ such that $G_{i_0}$ admits an embedding into $GL(N_1+1)\times GL(N_2+1)$ in such a way that an $G_{i_0}$-equivariant embedding of $T_{i_0}$ into $\CP^{N_1}\times \CP^{N_2}$ exists.
Using the Segre map, we obtain for $N>0$ an embedding of $G_{i_0}$ into $\GL(N+1)$ and an $G_{i_0}$-equivariant embedding of $T_{i_0}$ into $\CP^N$.
Moreover, by Lemma \ref{lem:extendedaction}, $T_{i_0}$ contains a point $p_{i_0}$ which is fixed by $G_{i_0}$.
Hence, the construction method of Section \ref{subsec:keyconstr} can be applied to the above set of data.
Therefore, Proposition \ref{prop:constrmethod} yields an $n$-dimensional smooth complex projective variety $X$ whose primitive Hodge numbers $l^{p,q}(X)$ are given by formula (\ref{eq1:main1}).

For $p>q$ and $p+q<n$, the $G_{i_0}$-invariant cohomology of $T_{i_0}$ is trivial whenever $(p,q)$ is different from $(n-2,0)$ and $(n-1,0)$.
Moreover, for $(p,q)= (n-1,0)$ it has dimension $l^{n-1,0}$ and for $(p,q)= (n-2,0)$ its dimension equals $l^{n-2,0}$.
Thus, (\ref{eq1:main1}) and the Hodge symmetries on $X$ yield $l^{p,q}(X)=l^{p,q}$ for all $p\neq q$ with $p+q<n$.
Moreover, as in case (\ref{item1:proof:main1}), we obtain 
\[
l^{p,p}(X)=l^{p,p}+C_2(p,n) \ ,
\]
where $C_2(p,n)$ is a constant in $p$ and $n$ which can be estimated by
\[
C_2(p,n)\leq p+1+\sum_{\substack{r\geq s>0 \\ n>r+s>2}} p \ ,
\]
where we used that $H^{p,p}(T_{i_0})^{G_{i_0}}$ has dimension $p+1$.
Our estimation for $C_1(p,n)$ shows
\[
C_2(p,n)\leq p\cdot (n-1)^2/4 +p \ .
\]
Then, for $l^{p,p}\geq C_2(p,n)$, we may replace $l^{p,p}$ by $l^{p,p}-C_2(p,n)$ in the above argument and obtain an $n$-dimensional smooth complex projective variety with the given truncated Hodge diamond.

Let us now define
\begin{equation} \label{eq:C(p,n)}
C(p,n):=\max\left(C_1(p,n),C_2(p,n)\right) \ .
\end{equation}
Then in both cases, (\ref{item1:proof:main1}) and (\ref{item2:proof:main1}), a variety with the desired truncated Hodge diamond exists if $l^{p,p}\geq C(p,n)$.
Moreover, $C(p,n)$ can roughly be estimated by
\[
C(p,n)\leq p\cdot \frac{n^2-2n+5}{4} \ .
\]
This finishes the proof of Theorem \ref{thm:main1}.

\begin{remark}
As we have seen in the above proof, we may replace the given lower bound on $l^{p,p}$ in assumption (\ref{cond1:thm:main1}) of Theorem \ref{thm:main1} by the smaller constant $C(p,n)$, defined in (\ref{eq:C(p,n)}).
\end{remark}

\section{Special weight $2$ Hodge structures} \label{sec:deg2coho}
In this section we show that for weight two Hodge structures, the lower bound $h^{1,1}\geq 2$ in Theorem \ref{thm:main2} can be replaced by the optimal lower bound $h^{1,1}\geq 1$.
Our proof uses an ad hoc implementation of the Godeaux-Serre construction.
The examples we construct here compare nicely to the results in Appendices \ref{sec:3-folds} and \ref{sec:4-folds}.
However, since the methods of this section are not used elsewhere in the paper, the reader can easily skip this section.

\begin{theorem} \label{thm:main3}
Let $h^{2,0}$ and $h^{1,1}$ be natural numbers with $h^{1,1}\geq 1$.
Then in each dimension $\geq 3$ there exists a smooth complex projective variety $X$ with 
\[
h^{2,0}(X)=h^{2,0}\ \ \ \text{and}\ \ \ h^{1,1}(X)=h^{1,1}\ .
\]
\end{theorem}

\begin{proof}
Since blowing-up a point increases $h^{1,1}$ by one and leaves $h^{2,0}$ unchanged, in order to prove Theorem \ref{thm:main3}, it suffices to construct for given $g$ in each dimension $n>2$ a smooth complex projective variety $X$ with $h^{2,0}(X)=g$ and $h^{1,1}(X)=1$.

We fix some large integers $N_1$ and $N_2$ and consider $T:={C_g}^2$ together with the subgroups $G^1(2,0,g)$ and $G^2(2,0,g)$ of $\Aut(T)$, defined in Section \ref{subsec:groupaction}.
For $j=1,\ldots , N_1$, we denote a copy of $T^{N_2}$ by $A_j$ and we put
\[
A:=A_1\times \cdots \times A_{N_1}  \ .
\]
That is, $A$ is a $(2\cdot N_1\cdot N_2)$-fold product of $C_g$, but we prefer to think of $A$ to be an $N_1$-fold product of $T^{N_2}$, where the $j$-th factor is denoted by $A_j$.

Next, we explain the construction of a certain subgroup $G$ of automorphisms of $A$.
This group is generated by five finite subgroups $G_1,\ldots ,G_5$ in $\Aut(A)$. 
The first subgroup of $\Aut(A)$ is given by 
\[
G_1:=G^1(2,0,g)^{\times N_1} \ ,
\] 
where $G^1(2,0,g)$ acts on each $A_j$ via the diagonal action.
The second one is 
\[
G_2:=G^1(2,0,g)^{\times N_2} \ ,
\]
acting on $A$ via the diagonal action.
The third one is given by
\[
G_3:=G^2(2,0,g) \ ,
\]
acting on each $A_j$ as well as on $A$ via the diagonal action.
The fourth group of automorphisms of $A$ equals
\[
G_4:=\Sym(N_1) \ ,
\]
which acts on $A$ via permutation of the $A_j$'s.
Finally, we put
\[
G_5:=\Sym(N_2) \ ,
\]
which permutes the $T$-factors of each $A_j$ and acts on $A$ via the diagonal action.

Suppose we are given some elements $\phi_i\in G_i$.
Then, $\phi_3$ commutes with $\phi_4$ and $\phi_5$, and $\phi_3\circ\phi_1=\phi_1^\prime\circ \phi_3$, respectively $\phi_1\circ\phi_3=\phi_3\circ \phi_1^{\prime\prime}$ as well as $\phi_3\circ\phi_2=\phi_2^\prime\circ \phi_3$, respectively $\phi_2\circ\phi_3=\phi_3\circ \phi_2^{\prime\prime}$ holds for some $\phi_i^\prime,\phi_i^{\prime \prime}\in G_i$, where $i=1,2$.
Similar relations can be checked for all products $\phi_i\circ \phi_j$ and so we conclude that each element $\phi$ in the group $G\subseteq \Aut(A)$, which is generated by $G_1,\ldots ,G_5$, can be written in the form
\[
\phi=\phi_1\circ \phi_2 \circ \phi_3 \circ \phi_4\circ \phi_5 \ ,
\]
where $\phi_i$ lies in $G_i$.

Suppose that the fixed point set $\Fix_A(\phi)$ contains an irreducible component whose codimension is less than 
\[
\min\left(N_1/2,2N_2\right) \ .
\]
Since $\phi$ is just some permutation of the $2 N_1  N_2$ curve factors of $A$, followed by automorphisms of each factor, we deduce that $\phi$ needs to fix more than 
\[
2 N_1  N_2 -\min(N_1,4N_2)
\] 
curve factors.
If $\phi_4$ were nontrivial, then $\phi$ would fix at most $2(N_1-2)N_2$ curve factors, and if $\phi_5$ were nontrivial, then $\phi$ would fix at most $2N_1(N_2-2)$ curve factors.
Thus, $\phi_4=\phi_5=\id$.
If $\phi_3$ were nontrivial, then its action on a single factor $T={C_g}^2$ cannot permute the two curve factors.
Thus, $\phi_3$ is just multiplication with $-1$ on each curve factor.
This cannot be canceled with automorphisms in $G^1(2,0,g)$, since the latter is a cyclic group of order $2g+1$.
Therefore, $\phi_3=\id$ follows as well.

Since $\phi$ fixes more than $2N_1N_2-N_1$ curve factors, we see that $\phi=\phi_1\circ \phi_2$ needs to be the identity on at least one $A_{j_0}$.
Since $\phi_2$ acts on each $A_j$ in the same way, it lies in $G_1\cap G_2$ and so we may assume $\phi_2=\id$.
Finally, any nontrivial automorphism in $G_1$ has a fixed point set of codimension $\geq 2N_2$.
This is a contradiction. 

For $N_1$ and $N_2$ large enough, it follows that the $G$-action on $A$ is free outside a subset of codimension $>n$.
Then, by Lemma \ref{lem:Sommese}, $A/G$ contains a smooth $n$-dimensional subvariety $X$ whose cohomology below degree $n$ is given by the $G$-invariants of $A$.

For the proof of the Theorem, it remains to show $h^{2,0}(X)=g$ and $h^{1,1}(X)=1$.
For this purpose, we denote the fundamental class of the $j$-th curve factor of $A$ by 
\[
\Omega_j\in H^{1,1}(A) \ .
\]
Moreover, we pick for $j=1,\ldots ,2N_1N_2$ a basis $\omega_{j1},\ldots ,\omega_{jg}$ of $(1,0)$-classes of the $j$-th curve factor of $A$ in such a way that 
\[
\psi_g^\ast\omega_{jl}=\zeta^{l}\omega_{jl} \ ,
\]
for a fixed $(2g+1)$-th root of unity $\zeta$ holds.
Then the cohomology ring of $A$ is generated by the $\Omega_j$'s, $\omega_{jl}$'s and their conjugates.

Suppose that we are given a $G$-invariant $(1,1)$-class which contains $\omega_{is}\wedge\overline{\omega_{jr}}$ nontrivially.
Then application of a suitable automorphism in $G_1$ shows that after relabeling $A_1,\ldots ,A_{N_1}$, we may assume $1\leq i,j\leq 2N_2$.
Moreover, it follows that $i$ and $j$ have the same parity, since otherwise $r+s$ is zero modulo $2g+1$, which contradicts $1\leq r,s\leq g$.
Finally, application of a suitable element in $G_2$ shows $i=j$.
Since $\omega_{is}\wedge\overline{\omega_{ir}}$ is a multiple of $\Omega_i$, it follows that our $G$-invariant $(1,1)$-class is 
of the form
\[
\lambda_1\cdot \Omega_1+\cdots +\lambda_{2N_1N_2}\cdot \Omega_{2N_1N_2} \ .
\]
Since $G$ acts transitively on the curve factors of $A$, this class is $G$-invariant if and only if $\lambda_1=\cdots =\lambda_{2N_1N_2}$.
This proves $h^{1,1}(X)=1$.

It remains to show $h^{2,0}(X)=g$.
Therefore, we define for $l=1,\ldots ,g$ the $(2,0)$-class
\[
\omega_l:=\sum_{i=1}^{N_1N_2}\omega_{2i-1\:l}\wedge\omega_{2i\:l}
\]
and claim that these form a basis of the $G$-invariant $(2,0)$-classes of $A$.
Clearly, they are linearly independent and it is easy to see that they are $G$-invariant.

Conversely, suppose that a $G$-invariant class contains $\omega_{il_1}\wedge \omega_{jl_2}$ nontrivially.
Then, application of a suitable element in $G_1$ shows that $l_1\pm l_2$ is zero modulo $2g+1$. 
This implies $l_1=l_2$. 
Therefore, our $G$-invariant $(2,0)$-class is of the form 
\[
\sum_{ijl} \lambda_{ijl}\cdot \omega_{il}\wedge \omega_{jl} \ .
\]
For fixed $l=1,\ldots ,g$, we write $\lambda_{ij}=\lambda_{ijl}$ and note that 
\[
\sum_{ij} \lambda_{ij}\cdot \omega_{il}\wedge \omega_{jl} 
\]
is also $G$-invariant.
We want to show that this class is a multiple of $\omega_l$. 
Applying suitable elements of $G_1$ shows that the above $(2,0)$-class is a sum of $(2,0)$-classes of the factors $A_1,\ldots ,A_{N_1}$.
Since this sum is invariant under the permutation of the factors $A_1,\ldots ,A_{N_1}$, it suffices to consider the class
\[
\sum_{i,j=1}^{2N_2} \lambda_{ij}\cdot \omega_{il}\wedge \omega_{jl}
\]
on $A_1$, which is invariant under the induced $G_2$- and $G_5$-action on $A_1$.
In this sum we may assume $\lambda_{ij}=0$ for all $i\geq j$ and application of a suitable element in $G_2$ shows that the above class is given by
\[
\sum_{i=1}^{N_2}\lambda_{2i-1\:2i}\cdot \omega_{2i-1\:l}\wedge \omega_{2i\:l} \ .
\]
Finally, application of elements of $G_5$ proves that our class is a multiple of 
\[
\sum_{i=1}^{N_2} \omega_{2i-1\:l}\wedge \omega_{2i\:l} \ .
\]
This finishes the proof of $h^{2,0}(X)=g$ and thereby establishes Theorem \ref{thm:main3}.
\end{proof}

\begin{remark}
The above construction does not generalize to higher degrees -- at least not in the obvious way.
\end{remark}

\section{Primitive Hodge numbers away from the vertical middle axis} \label{sec:Z_n}
In this section we produce examples whose primitive Hodge numbers away from the vertical middle axis of the Hodge diamond (\ref{eq:diamond}) are concentrated in a single $(p,q)$-type.
These examples will then be used in the proof of Theorem \ref{thm:domination} in Section \ref{sec:domination}.
Our precise result is as follows:

\begin{theorem} \label{thm:Z_n}
For $a>b\geq 0$, $n\geq a+b$ and $c\geq 1$, there exists an $n$-dimensional smooth complex projective variety whose primitive $(p,q)$-type cohomology has dimension $(3^c-1)/2$ if $p=a$ and $q=b$, and vanishes for all other $p>q$.
\end{theorem}

In comparison with Theorem \ref{thm:main1}, the advantage of Theorem \ref{thm:Z_n} is that it also controls the Hodge numbers $h^{p,q}$ with $p\neq q$ and $p+q=n$.
These numbers lie in the horizontal middle row of the Hodge diamond (\ref{eq:diamond}) and so they were excluded in the statement of Theorem \ref{thm:main1}.

Using an iterated resolution of $(\Z/3\Z)$-quotient singularities whose local description is given in Section \ref{subsec:localres}, we explain an inductive construction method in Section \ref{subsec:ind.constr}. 
Using this construction, Theorem \ref{thm:Z_n} will easily follow in Section \ref{subsec:proof:thm:Z_n}.
Our approach is inspired by Cynk--Hulek's construction of rigid Calabi-Yau manifolds \cite{cynk&hulek}. 

\subsection{Local resolution of $\Z/3\Z$-quotient singularities} \label{subsec:localres}
Fix a primitive third root of unity $\xi$ and choose affine coordinates $(x_1,\ldots ,x_n)$ on $\C^n$.
For an open ball $Y\subseteq \C^n$ centered at $0$ and for some $r\geq 0$, we consider the automorphism $\phi:Y \to Y$ given by
\[
(x_1,\ldots ,x_n)\mapsto (\xi \cdot x_1,\ldots ,\xi \cdot x_r,\xi^2 \cdot x_{r+1},\ldots ,\xi^2 \cdot x_n) \ .
\]
Let $Y^\prime$ be the blow-up of $Y$ in the origin with exceptional divisor $E^\prime\subseteq Y^\prime$.
Then $\phi$ lifts to an automorphism $\phi'\in \Aut(Y^\prime)$ and we define $Y^{\prime\prime}$ to be the blow-up of $Y^\prime$ along $\Fix_{Y^{\prime}}(\phi')$.
The exceptional divisor of this blow-up is denoted by $E^{\prime\prime}\subseteq Y^{\prime\prime}$ and $\phi'$ lifts to an automorphism $\phi''\in \Aut(Y'')$.
In this situation, we have the following lemma.

\begin{lemma} \label{lem:sing1}
The fixed point set of $\phi''$ on $Y^{\prime\prime}$ equals $E^{\prime\prime}$.
Moreover:
\begin{enumerate}
\item If $r=0$ or $r=n$, then $E^{\prime\prime}\cong E^\prime \cong \CP^{n-1}$.
Otherwise, $E^\prime\cong \CP^{n-1}$ and $E^{\prime\prime}$ is a disjoint union of $\CP^{r-1}\times \CP^{n-r}$ and $\CP^r\times \CP^{n-r-1}$.
\label{item1:sing1}
\item The quotient $Y^{\prime\prime}/\phi''$ is smooth and admits local holomorphic coordinates $(z_1,\ldots,z_n)$ where each $z_j$ comes from a $\phi$-invariant meromorphic function on $Y$, explicitly given by a quotient of two monomials in $x_1,\ldots , x_n$.\label{item2:sing1} 
\end{enumerate}
\end{lemma}

\begin{proof}
This Lemma is proven by a calculation, similar to that in \cite[pp.\ 84-87]{kollar:sing}, where the case $n=2$ is carried out.

The automorphism $\phi'$ acts on the exceptional divisor $E^{\prime }\cong \CP^{n-1}$ of $Y^\prime \to Y$ as follows:
\[
[x_1:\ldots :x_n]\mapsto [\xi \cdot x_1:\ldots :\xi \cdot x_r:\xi^2 \cdot x_{r+1}:\ldots :\xi^2 \cdot x_{n}] \ .
\]
Hence, if $r=0$ or $r=n$, then $\Fix_{Y^\prime}(\phi')$ equals $E^\prime$.
Since this is a smooth divisor on $Y^\prime$, the blow-up $Y^{\prime\prime}\to Y^\prime$ is an isomorphism and the quotient $Y^{\prime\prime}/\phi''$ is smooth.
Moreover, $E^\prime\cong E^{\prime \prime}$ is covered by $n$ charts $U_1,\ldots ,U_n$ such that on $U_i$, coordinates are given by
\begin{align} \label{eq:blow-up:coord}
\left(\frac{x_1}{x_i},\ldots , \frac{x_{i-1}}{x_i},x_i,\frac{x_{i+1}}{x_i},\ldots ,\frac{x_n}{x_i}\right) \ .
\end{align}
The quotient $Y^{\prime\prime}/\phi''$ is then covered by $U_1/\phi'',\ldots ,U_n/\phi''$.
Coordinate functions on $U_i/\phi''$ are given by the following $\phi$-invariant rational functions on $Y$:
\[
\left(\frac{x_1}{x_i},\ldots , \frac{x_{i-1}}{x_i},x_i^3,\frac{x_{i+1}}{x_i},\ldots ,\frac{x_n}{x_i}\right) \ .
\]
This proves the Lemma for $r=0$ or $r=n$.

If $0<r<n$, then $\Fix_{Y^\prime}(\phi')$ equals the disjoint union of $E^\prime_1\cong \CP^{r-1}$ and $E^\prime_2\cong \CP^{n-r-1}$, sitting inside $E^\prime$.
The exceptional divisor $E^\prime$ is still covered by the $n$-charts $U_1,\ldots ,U_n$, defined above.
Moreover, the charts $U_1,\ldots ,U_r$ cover $E^\prime_1$ and $U_{r+1},\ldots ,U_n$ cover $E_2^\prime$. 
Fix a chart $U_i$ with coordinate functions $(z_1,\ldots ,z_n)$.
If $i\leq r$, then $\phi'$ acts on $r-1$ of these coordinates by the identity and on the remaining coordinates by multiplication with $\xi$.
Conversely, if $i>r$, then $\phi'$ acts on $n-r-1$ coordinates by the identity and on the remaining coordinates by multiplication with $\xi^2$.
We are therefore in the situation discussed in the previous paragraph and the Lemma follows by an application of that result in dimension $n-r+1$ and $r+1$ respectively.
\end{proof}

\subsection{Inductive approach} \label{subsec:ind.constr}
In this section we explain a general construction method which will allow us to prove Theorem \ref{thm:Z_n} by induction on the dimension in Section \ref{subsec:proof:thm:Z_n}.

For natural numbers $a\neq b$ and $c\geq 0$, let $\mathcal S_c^{a,b}$ denote the family of pairs $(X,\phi)$, consisting of a smooth complex projective variety $X$ of dimension $a+b$ and an automorphism $\phi\in \Aut(X)$ of order $3^c$, such that properties (\ref{item:h^{a,b}})--(\ref{item:coho(Fix)}) below hold.
Here, $\zeta$ denotes a fixed primitive $3^c$-th root of unity and $g:=(3^c-1)/2$:
\begin{enumerate}
\item The Hodge numbers $h^{p,q}$ of $X$ are given by $h^{a,b}=h^{b,a}=g$ and $h^{p,q}=0$ for all other $p\neq q$.\label{item:h^{a,b}}
\item The action of $\phi$ on $H^{a,b}(X)$ has eigenvalues $\zeta,\ldots, \zeta^g$. \label{item:eig}
\item The group $H^{p,p}(X)$ is for all $p\geq 0$ generated by algebraic classes which are fixed by the action of $\phi$. \label{item:alg.classes}
\item The set $\Fix_X\left(\phi^{3^{c-1}}\right)$ can be covered by local holomorphic charts such that $\phi$ acts on each coordinate function by multiplication with some power of $\zeta$.\label{item:Fix}
\item For $0\leq l\leq c-1$, the cohomology of $\Fix_X\left(\phi^{3^l}\right)$ is generated by algebraic classes which are fixed by the action of $\phi$. \label{item:coho(Fix)}
\end{enumerate}
For $0\leq l\leq c-1$, we have obvious inclusions 
\[
 \Fix_X\left(\phi^{3^l}\right)\subseteq \Fix_X\left(\phi^{3^{c-1}}\right) .
\] 
It therefore follows from (\ref{item:Fix}) that $\Fix_X\left(\phi^{3^l}\right)$ can be covered by local holomorphic coordinates on which $\phi^{3^l}$ acts by multiplication with some power of $\zeta^{3^l}$. 
In particular, $\Fix_X\left(\phi^{3^l}\right)$ is smooth for all $0\leq l \leq c-1$; its cohomology is of $(p,p)$-type, since it is generated by algebraic classes by (\ref{item:coho(Fix)}).
We also remark that condition (\ref{item:alg.classes}) implies that each variety in $\mathcal S_c^{a,b}$ satisfies the Hodge conjecture.
Finally, note that $(X,\phi)\in \mathcal S_c^{a,b}$ is equivalent to $(X,\phi^{-1})\in \mathcal S_c^{b,a}$.

The inductive approach to Theorem \ref{thm:Z_n} is now given by the following.

\begin{proposition} \label{prop:ind.constr}
Let $(X_1,\phi_1^{-1})\in  \mathcal S_c^{a_1,b_1}$ and $(X_2,\phi_2)\in \mathcal S_c^{a_2,b_2}$.
Then 
\[
\left(X_1\times X_2\right) / \left\langle \phi_1\times \phi_2\right\rangle
\] 
admits a smooth model $X$ such that the automorphism $\id \times \phi_2$ on $X_1\times X_2$ induces an automorphism $\phi\in \Aut(X)$ with $(X,\phi)\in \mathcal S_c^{a,b}$, where $a=a_1+a_2$ and $b=b_1+b_2$.
\end{proposition} 

\begin{proof}
We define the subgroup 
\[
G:=\left\langle \phi_1\times \id, \id \times \phi_2\right\rangle  
\]
of $\Aut(X_1\times X_2)$.
For $i=1,\ldots ,c$ we consider the element
\[
\eta_i:=\left(\phi_1\times \phi_2\right)^{3^{c-i}} 
\]
of order $3^i$ in $G$.
This element generates a cyclic subgroup
\[
G_i:=\left\langle \eta_i\right\rangle \subseteq G \ ,
\]
and we obtain a filtration
\[
0=G_0\subset G_1\subset \cdots \subset G_c=\left\langle \phi_1\times \phi_2\right\rangle \ ,
\]
such that each quotient $G_i/G_{i-1}$ is cyclic of order three, generated by the image of $\eta_i$.

By definition, $G$ acts on 
\[
Y_0:=X_1\times X_2 \ .
\] 
Using the assumptions that $(X_1,\phi_1^{-1})$ and $(X_2,\phi_2)$ satisfy (\ref{item:h^{a,b}})--(\ref{item:alg.classes}), it is easily seen (and we will give the details later in this proof) that the $\left\langle \phi_1\times \phi_2\right\rangle$-invariant cohomology of $Y_0$ has Hodge numbers $h^{a,b}=h^{b,a}=g$ and $h^{p,q}\neq 0$ for all other $p\neq q$.
The strategy of the proof of Proposition \ref{prop:ind.constr} is now as follows.

We will construct inductively for $i=1,\ldots ,c$ smooth models $Y_i$ of $Y_0/G_i$, fitting into the following diagram:
\begin{align} \label{eq:diag:Y_i}
\begin{xy}
  \xymatrix{
		 &  \ar[dl] Y_{c-1}^{\prime \prime} \ar[dr]   &  	 				& \ar[dl] \cdots \ar[dr] & 				&  \ar[dl] Y_1^{\prime \prime} \ar[dr]	& 
		 & Y_0^{\prime \prime} \ar[rd] \ar[ld]&\\
Y_c &  																					 &	 Y_{c-1}	&												& Y_2 	&																				& Y_1
		& & Y_0  \ .
  }	
\end{xy}
\end{align}
Here, $Y_{i-1}^{\prime \prime}\rightarrow Y_{i}$ will be a $3:1$ cover, branched along a smooth divisor, and  $Y_i''\to Y_i$ will be the composition $Y_i''\to Y_i'\to Y_i$ of two blow-down maps. 
This way we obtain a smooth model 
\[
X:=Y_c
\] 
of $Y_0/\left\langle \phi_1\times \phi_2\right\rangle$.
At each stage of our construction, the group $G$ will act (in general non-effectively) and we will show that each blow-up and each triple quotient changes the $\left\langle \phi_1\times \phi_2\right\rangle$-invariant cohomology only by algebraic classes which are fixed by the $G$-action.
Since $\left\langle \phi_1\times \phi_2\right\rangle$ acts trivially on $X$, it follows that $H^\ast(X,\C)$ is generated by $\left\langle \phi_1\times \phi_2\right\rangle$-invariant classes on $Y_0$ together with algebraic classes which are fixed by the action of $G$.
Hence, $X$ satisfies (\ref{item:h^{a,b}}).
We then define $\phi\in \Aut(X)$ via the action of $\id\times \phi_2 \in G$ on $Y_c$ and show carefully that the technical conditions (\ref{item:eig})--(\ref{item:coho(Fix)}) are met by $(X,\phi)$.

In the following, we give the details of the approach outlined above.

We begin with the explicit construction of diagram (\ref{eq:diag:Y_i}).
Firstly, let $Y_0^{\prime}$ be the blow-up of $Y_0$ along $\Fix_{Y_0}\left(\eta_1\right)$.
Since $G$ is an abelian group, its action on $Y_0$ restricts to an action on $\Fix_{Y_0}\left(\eta_1\right)$ and so it lifts to an action on the blow-up $Y_0'$.
This allows us to define $Y_0^{\prime\prime}$ via the blow-up of $Y_0^{\prime}$ along $\Fix_{Y_0^\prime}\left(\eta_1\right)$.
Again, $G$ lifts to $Y_0^{\prime\prime}$ since it is abelian.
Using this action, we define 
\[
Y_1:=Y_0^{\prime \prime}/\left\langle \eta_1\right\rangle \ ,
\]
where by abuse of notation, $\left\langle \eta_1\right\rangle$ denotes the subgroup of $\Aut(Y_0'')$ which is generated by the action of $\eta_1\in G$.

We claim that $Y_1$ is a smooth model of $Y_0/\left\langle \eta_1 \right\rangle$.
To see this, we define 
\[
U_0:=Y_0\setminus \Fix_{Y_0}(\eta_1) 
\]
and note that the preimage of this set under the blow-down maps
\[
Y_0''\longrightarrow Y_0' \longrightarrow Y_0
\] 
gives Zariski open subsets
\[
U_0'\subseteq Y_0'\ \ \text{and}\ \ U_0''\subseteq Y_0'' \ ,
\]
both isomorphic to $U_0$.
The group $G$ acts on these subsets and so
\[
U_1:=U_0''/\left\langle \eta_1\right\rangle
\] 
is a Zariski open subset in $Y_1$ which is isomorphic to the Zariski open subset
\[
U_0/\left\langle \eta_1 \right\rangle \subseteq Y_0/\left\langle \eta_1 \right\rangle \ . 
\] 
The latter is smooth since $\eta_1$ acts freely on $U_0$ and so it remains to see that $Y_1$ is smooth at points of the complement of $U_1\subseteq Y_1$.
To see this, note that by (\ref{item:Fix}), 
\[
\Fix_{Y_0}\left(\eta_1\right)=\Fix_{X_1}(\phi_1^{3^{c-1}})\times \Fix_{X_2}(\phi_2^{3^{c-1}}) 
\]
inside $Y_0$ can be covered by local holomorphic coordinates on which $\phi_1\times \phi_2$ acts by multiplication with some powers of $\zeta$.
On these coordinates, $\eta_1$ acts by multiplication with some powers of a third root of unity.
The local considerations of Lemma \ref{lem:sing1} therefore apply and we deduce that $Y_1$ is indeed a smooth model of $Y_0/G_1$.

Since $G$ is abelian, the $G$-action on $Y_0''$ descends to a $G$-action on $Y_1$.
The subgroup $G_1\subseteq G$ acts trivially on $Y_1$ and the induced $G/G_1$-action on $Y_1$ is effective.
Also note that $G_i$ acts freely on $U_0\subseteq Y_0$ and so $G_i/G_1$ acts, for $2\leq i\leq c$, freely on the Zariski open subset $U_1\subseteq Y_1$. 
By (\ref{item:Fix}), the complement of $U_0$ in $Y_0$ can be covered by local holomorphic coordinates on which $G$ acts by multiplication with some roots of unity on each coordinate.
It therefore follows from the second statement in Lemma \ref{lem:sing1} that the complement of $U_1$ in $Y_1$ can also be covered by local holomorphic coordinates in which $G$ acts by multiplication with some roots of unity on each coordinate.
This shows that we can repeat the above construction inductively.

We obtain for $i\in \left\{1,\ldots ,c\right\}$ smooth models
\[
Y_i:=Y_{i-1}''/\left\langle \eta_i\right\rangle
\]
of $Y_0/G_i$ on which $G$ acts (non-effectively). 
The smooth model $Y_i$ contains a Zariski open subset
\[
U_i\cong U_0/\left\langle \eta_i\right\rangle
\]
on which $G_l/G_i$ acts freely for all $i+1\leq l \leq c$; explicitly, $U_i:= U_{i-1}''/\left\langle \eta_i\right\rangle$, where $U_{i-1}''\subseteq Y_{i-1}''$ is isomorphic to $U_{i-1}$.
The complement of $U_i$ is covered by local holomorphic coordinates on which $G$ acts by multiplication with some roots of unity on each coordinate.

$Y_i''$ is then defined via the two-fold blow-up
\begin{align} \label{eq:blow-upseq}
Y_i''\longrightarrow Y_i'\longrightarrow Y_i \ ,
\end{align}
where one blows up the fixed point set of the action of $\eta_{i+1}$ on $Y_i$ and $Y_i'$ respectively.
The preimage of $U_i$ in $Y_i'$ and $Y_i''$ gives Zariski open subsets
\[
U_i'\subseteq Y_i'\ \ \text{and}\ \ U_i''\subseteq Y_i'' \ ,
\]
which are both isomorphic to $U_i$.
Since $G$ is abelian, the $G$-action on $Y_i$ induces actions on $Y_i'$ and $Y_i''$ and these actions restrict to actions on $U_i\cong U_i'\cong U_i''$. 
The complement of $U_i'$ in $Y_i'$ (resp.\ $U_i''$ in $Y_i''$) is by Lemma \ref{lem:sing1} covered by local holomorphic coordinates on which $G$ acts by multiplication with some roots of unity on each coordinate.
Using the local considerations in Lemma \ref{lem:sing1}, it follows that $Y_{i+1}=Y_i''/\left\langle \eta_{i+1}\right\rangle$ is a smooth model of $Y_0/G_{i+1}$ which has the above stated properties.
This finishes the inductive construction of diagram (\ref{eq:diag:Y_i}).

Our next aim is to compute the cohomology of $Y_c$.
Since $G_c$ acts trivially on $Y_c$, we may as well compute the $G_c$-invariant cohomology of $Y_c$.
This point of view has the advantage that it allows an inductive approach, since for $i=0,\ldots ,c-1$, the $G_c$-invariant cohomology of $Y_i$ is easier to compute than its ordinary cohomology.

Before we can carry out these calculations, we have to study the action of arbitrary subgroups $\Gamma \subseteq G$ on $Y_i$, $Y_i'$ and $Y_i''$.
Since $G$ is an abelian group, it follows that it acts on the fixed point sets $\Fix_{Y_i}(\Gamma)$, $\Fix_{Y_i'}(\Gamma)$ and $\Fix_{Y_i''}(\Gamma)$, defined in (\ref{def:Fix_{Y}(Gamma)}). 
These actions have the following important properties, where as usual, cohomology means singular cohomology with coefficients in $\C$ (see our conventions in Section \ref{subsec:conv}).

\begin{lemma} \label{lem:Fix(Gamma)}
Let $\Gamma\subseteq G$ be a subgroup which is not contained in $G_i$.
Then $\Fix_{Y_i}(\Gamma)$, $\Fix_{Y_i'}(\Gamma)$ and $\Fix_{Y_i''}(\Gamma)$ are smooth, their $G$-actions restrict to actions on each irreducible component and their $G_c$-invariant cohomology is generated by $G$-invariant algebraic classes.
\end{lemma}
Note that the assumption $\Gamma \nsubseteq G_i$ is equivalent to saying that the action of $\Gamma$ is nontrivial on each of the spaces $Y_i$, $Y_i'$ and $Y_i''$.

\begin{proof}[Proof of Lemma \ref{lem:Fix(Gamma)}]
To begin with, we want to verify the Lemma for $\Fix_{Y_0}(\Gamma)$, where $\Gamma\subseteq G$ is nontrivial.
Recall that $Y_0=X_1\times X_2$ and that each element in $\Gamma$ is of the form $\phi_1^j\times \phi_2^k$.
The fixed point set of such an element is then given by
\[
\Fix_{Y_0}(\phi_1^j\times \phi_2^k)=\Fix_{X_1}(\phi_1^j)\times \Fix_{X_2}(\phi_2^k) \ .
\]
The intersection of sets of the above form is still of the above form and so
\[
\Fix_{Y_0}(\Gamma)=\Fix_{X_1}(\phi_1^j)\times \Fix_{X_2}(\phi_2^k) \ ,
\]
for some natural numbers $j$ and $k$.
Since $(X_1,\phi_1^{-1})$ and $(X_2,\phi_2)$ satisfy (\ref{item:Fix}), it follows that $\Fix_{Y_0}(\Gamma)$ is smooth.
Also, $G$ acts trivially on $H^0(\Fix_{Y_0}(\Gamma),\C)$ by (\ref{item:coho(Fix)}) and so the $G$-action restricts to an action on each irreducible component of $\Fix_{Y_0}(\Gamma)$.

Since $\Gamma$ is not the trivial group, we now assume without loss of generality that $j$ is not divisible by $3^c$.
Since $(X,\phi_1^{-1})$ satisfies (\ref{item:coho(Fix)}), the cohomology of $\Fix_{X_1}(\phi_1^j)$ is then generated by $\left\langle \phi_1\right\rangle$-invariant algebraic classes.
The $G_c$-invariant cohomology of $\Fix_{Y_0}(\Gamma)$ is therefore generated by products of these algebraic classes with $\left\langle \phi_2\right\rangle$-invariant classes on $\Fix_{X_2}(\phi_2^k)$.
Since $(X_2,\phi_2)$ satisfies (\ref{item:h^{a,b}})--(\ref{item:alg.classes}) and (\ref{item:coho(Fix)}), the latter are, regardless whether $k$ is divisible by $3^c$ or not, given by $\left\langle \phi_2\right\rangle$-invariant algebraic classes.
This shows that the $G_c$-invariant cohomology of $\Fix_{Y_0}(\Gamma)$ is generated by $G$-invariant algebraic classes, as we want.

Using induction, let us now assume that the Lemma is true for $\Fix_{Y_i}(\Gamma)$ for some $i\geq 0$ and for all $\Gamma \nsubseteq G_i$.
Blowing-up $\Fix_{Y_i}(\eta_{i+1})$ on $Y_i$, we obtain the following diagram: 
\begin{align*} 
\begin{xy}
  \xymatrix{
	   \ar [d] \Fix_{Y_i'}(\Gamma)	  \ar@{^{(}->}[r] & \ar[d] Y_{i}' 	\\
				\Fix_{Y_i}(\Gamma)	\ar@{^{(}->}[r] & Y_i 			
  }	
\end{xy}
\end{align*}
and we denote the exceptional divisor of the blow-up $Y_i'\to Y_i$ by $E_i'\subseteq Y_i'$.

Let us first prove that $\Fix_{Y_i'}(\Gamma)$ is smooth and that $G$ acts on its irreducible components.
To see this, note that away from $E_i'$, the blow-down map $Y_i'\to Y_i$ is an isomorphism onto its image.
Since $\Fix_{Y_i}(\Gamma)$ is smooth, it is then clear that the intersection of $\Fix_{Y_i'}(\Gamma)$ with $Y_i'\setminus E_i'$ is smooth.
Also, $G$ acts on the irreducible components of $\Fix_{Y_i'}(\Gamma)$ which are not contained in $E_i'$, since the analogous statement holds for the components of $\Fix_{Y_i}(\Gamma)$.
On the other hand, $E_i'$ can be covered by local holomorphic coordinates on which $G$ acts by multiplication with roots of unity.
In each of these charts, $\Fix_{Y_i'}(\Gamma)$ corresponds to a linear subspace on which $G$ acts.
We conclude that $\Fix_{Y_i'}(\Gamma)$ is smooth and that $G$ acts on each of its irreducible components. 

Next, let $P$ be an irreducible component of $\Fix_{Y_i'}(\Gamma)$.
We have to prove the following

\begin{claim}
The $G_c$-invariant cohomology of $P$ is generated by $G$-invariant algebraic classes.
\end{claim}

\begin{proof}
Let us denote the image of $P$ in $Y_i$ by $Z$.
Then $Z$ is contained in $\Fix_{Y_i}(\Gamma)$ and the proof of the claim is divided into two cases.

In the first case, we suppose that $Z$ is not contained in the intersection 
\begin{align} \label{eq:Fix(Gamma,eta)}
\Fix_{Y_i}(\left\langle \Gamma,\eta_{i+1}\right\rangle)=\Fix_{Y_i}(\Gamma)\cap \Fix_{Y_i}(\eta_{i+1}) \ .
\end{align}
In this case, $P$ is the strict transform of $Z$ in $Y_i'$.
Conversely, if $\tilde Z \subseteq \Fix_{Y_i}(\Gamma)$ is any irreducible component, not contained in (\ref{eq:Fix(Gamma,eta)}), then its strict transform in $Y_i'$ is contained in $\Fix_{Y_i'}(\Gamma)$.
Hence, $Z$ is in fact an irreducible component of $\Fix_{Y_i}(\Gamma)$.
This implies that $\Fix_{Z}(\eta_{i+1})$ consists of irreducible components of (\ref{eq:Fix(Gamma,eta)}) and so $\Fix_{Z}(\eta_{i+1})$ is smooth by induction. 
Moreover, the strict transform $P$ of $Z$ in $Y_i'$ can be identified with the blow-up of $Z$ along $\Fix_{Z}(\eta_{i+1})$.
We denote the exceptional divisor of this blow-up by $D$ and obtain natural maps
\[
f:D \hookrightarrow P\ \ \text{and}\ \ g:D\rightarrow \Fix_{Z}(\eta_{i+1}) \ ,
\]
where $f$ denotes the inclusion and $g$ the projection map respectively.
Using Theorem \ref{thm:blow-up} and (\ref{eq:ringstr:blow-up}), we see that the cohomology of $P$ is generated (as a $\C$-module) by pull-back classes of $Z$ together with products 
\[
[D']^j\wedge f_\ast(g^\ast(\alpha)) \ ,
\]
where $D'$ is an irreducible component of $D$, $j$ is some natural number and $\alpha$ is a cohomology class on $\Fix_{Z}(\eta_{i+1})$.

The image $g(D')$ is an irreducible component of $\Fix_{Z}(\eta_{i+1})$.
By induction, $G$ acts on $g(D')$ and hence also on $D'$, the projectivization of the normal bundle of $g(D')$ in $Z$.
This implies that $[D']\in H^\ast(P,\C)$ is a $G$-invariant algebraic class.
Moreover, the $G_c$-invariant cohomology of $Z$ as well as the $G_c$-invariant cohomology of $\Fix_Z(\eta_{i+1})$ is generated by $G$-invariant algebraic classes by induction.
It therefore follows from the above description of $H^\ast(P,\C)$ that the $G_c$-invariant cohomology of $P$ is indeed generated by $G$-invariant algebraic classes.

It remains to deal with the case where the image $Z$ of $P$ in $Y_i$ is contained in (\ref{eq:Fix(Gamma,eta)}).
In this case, around each point of $Z$ there are local holomorphic coordinates $(z_1,\ldots, z_n)$ on which $G$ acts by multiplication with some roots of unity.
In these local coordinates, the fixed point set of $\eta_{i+1}$ corresponds to the vanishing set of certain coordinate functions.
After relabeling these coordinate functions if necessary, we may therefore assume that locally, $\Fix_{Y_i}(\eta_{i+1})$ corresponds to $\left\{z_{m}=\cdots =z_n=0\right\}$ for some $ m\leq n$.
This yields local homogeneous coordinates
\begin{align} \label{eq:coord}
(z_1,\ldots,z_{m-1},[z_{m}:\cdots :z_n]) 
\end{align}
along the exceptional divisor $E_i'$ of $Y_i'\to Y_i$.
After relabeling of the first $m-1$ coordinates if necessary, we may assume that $\Gamma$ acts trivially on $z_1,\ldots, z_{k-1}$ and nontrivially on $z_{k},\ldots, z_{m-1}$ for some $1\leq k \leq m-1$.
After relabeling $z_m,\ldots,z_n$ if necessary, we may then assume that in the homogeneous coordinates (\ref{eq:coord}), $P$ corresponds to $\left\{z_k=\cdots =z_{h}=0\right\}$ for some $m\leq h\leq n$.
Here, each element $\gamma\in \Gamma$ acts trivially on $[z_{h+1}:\ldots:z_n]$, that is, $\gamma$ acts by multiplication with the same root of unity on $z_{h+1},\ldots,z_n$. 

The above local description shows that $P\to Z$ is a $PGL$-subbundle of the $PGL$-bundle $E_i'|_Z\to Z$; explicit bundle charts for $P$ are given by $(z_1,\ldots,z_{k-1}, [z_{h+1}:\ldots:z_n])$ as above.
The exceptional divisor $E_i'$ carries the line bundle $\mathcal O_{E_i'}(1)$ and we denote its restriction to $P$ by $\mathcal O_P(1)$.
The cohomology of $P$ is then generated (as a $\C$-module) by products of pull-back classes on the base $Z$ with powers of $c_1(\mathcal O_P(1))$.
The line bundle $\mathcal O_{E_i'}(1)$ on the exceptional divisor $E_i'$ is isomorphic to the restriction of the line bundle $\mathcal O_{Y_i'}(-E_i')$ on $Y_i'$.
The first Chern class of the latter line bundle is $G$-invariant since $G$ acts on $E_i'$. 
It follows that $c_1(\mathcal O_P(1))$ is a $G$-invariant algebraic cohomology class on $P$.

In the above local coordinates $(z_1,\ldots,z_n)$ on $Y_i$, $Z$ is given by $\left\{z_{k}=\cdots = z_{n}=0\right\}$.
The latter set is in fact the fixed point set of $\left\langle \Gamma, \eta_{i+1}\right\rangle$ in this local chart and so it follows that $Z$ is an irreducible component of (\ref{eq:Fix(Gamma,eta)}).
By induction, the $G_c$-invariant cohomology of $Z$ is therefore generated by $G$-invariant algebraic classes.
By the above description of $H^\ast(P,\C)$, we conclude that the $G_c$-invariant cohomology of $P$ is generated by $G$-invariant algebraic classes, as we want.
This finishes the proof of our claim.
\end{proof}

Altogether, we see that the Lemma holds for $\Fix_{Y_i'}(\Gamma)$. 
Repeating the above argument, we then deduce the same assertion for $\Fix_{Y_i''}(\Gamma)$.

Next, let $\Gamma$ be a subgroup of $G$, not contained in $G_{i+1}$.
We denote by
\[
p_i:Y_i''\longrightarrow Y_{i+1}
\]
the quotient map.
Then, 
\[
p_i^{-1}(\Fix_{Y_{i+1}}(\Gamma)) = \left\{y\in Y''_{i}\ |\ g(y)\in \left\{y, \eta_{i+1}(y),\eta_{i+1}^2(y)\right\}\ \text{for all }g\in\Gamma \right\} \ .
\]
If this set is contained in $\Fix_{Y_i''}(\eta_{i+1})$, then it is given by $\Fix_{Y_i''}(\left\langle \Gamma, \eta_{i+1}\right\rangle)$.
The restriction of $p_i$ to $\Fix_{Y_i''}(\eta_{i+1})$ is an isomorphism onto its image and so we deduce that in this case, $\Fix_{Y_{i+1}}(\Gamma)$ satisfies the Lemma.

Conversely, if $p_i^{-1}(\Fix_{Y_{i+1}}(\Gamma))$ is not contained in $\Fix_{Y_i''}(\eta_{i+1})$, then we pick some 
\begin{align*} 
y\in p_i^{-1}(\Fix_{Y_{i+1}}(\Gamma))\ \ \text{with}\ \ y\notin \Fix_{Y_i''}(\eta_{i+1}) \ .
\end{align*}
Since $\eta_{i+1}$ acts trivially on $Y_{i+1}$ and since we are interested in $\Fix_{Y_{i+1}}(\Gamma)$, we assume without loss of generality that $\eta_{i+1}$ is contained in $\Gamma$.
Then, $\Gamma$ acts transitively on $\left\{y, \eta_{i+1}(y),\eta_{i+1}^2(y)\right\}$.
This gives rise to a short exact sequence
\[
1\longrightarrow H\longrightarrow \Gamma \longrightarrow \Z/3\Z \longrightarrow 1 \ ,
\]
where $H\subseteq \Gamma$ acts trivially on $y$ and where $g\in \Gamma$ is mapped to $j+ 3\Z$ if and only if $g(y)=\eta_{i+1}^j(y)$. 
Recall that $G\cong \Z/3^c\Z\times \Z/3^c\Z$, and so $\Gamma\cong \Z/3^k\Z \times \Z/3^m \Z$ for some $k,m\geq 0$. 
In the above short exact sequence, $\eta_{i+1}$ is mapped to a generator in $\Z/3\Z$ and so $\eta_{i+1}$ cannot be a multiple of $3$ in $\Gamma$.
That is,
\[
\Gamma\cong \left\langle \eta_{i+1}\right\rangle \times \left\langle \gamma\right\rangle\ ,
\]
for some $\gamma\in \Gamma$.
Since $\eta_{i+1}$ acts trivially on $Y_{i+1}$, one easily deduces
\begin{align} \label{eq:p_i^{-1}(Fix(Gamma))}
\Fix_{Y_{i+1}}(\Gamma)=\Fix_{Y_{i+1}}(\gamma)=\bigcup_{j=0}^2 p_i\left(\Fix_{Y_i''}(\gamma\circ \eta_{i+1}^{j}) \right) \ .
\end{align}
The irreducible components of $\Fix_{Y_{i+1}}(\Gamma)$ are therefore of the form $p_i(Z)$ where $Z$ is an irreducible component of 
\[
\bigcup_{j=0}^2 \Fix_{Y_i''}(\gamma\circ \eta_{i+1}^{j}) \ .
\]
As we have already proven the Lemma on $Y_i''$, we know that the $G$-action on $Y_i''$ restricts to an action on $Z$.
In particular, 
\[
p_i(Z)=Z/\left\langle \eta_{i+1}\right\rangle \ .
\] 
Since the abelian group $G$ acts on $Z$, it also acts on the above quotient.

For the moment we assume that $p_i(Z)$ is smooth.
Its cohomology is then given by the $\eta_{i+1}$-invariant classes on $Z$.
Since $\eta_{i+1}$ is contained in $G_c$, it follows that the $G_c$-invariant cohomology of $p_i(Z)$ is given by the $G_c$-invariant cohomology of $Z$.
Since we know the Lemma on $Y_i''$, the latter is generated by $G$-invariant algebraic classes, as we want.

It remains to see that $\Fix_{Y_{i+1}}(\Gamma)$ is smooth.
In the local holomorphic charts which cover the complement of $U_{i+1}$ in $Y_{i+1}$, this fixed point set is given by linear subspaces which are clearly smooth.
It therefore suffices to prove that the fixed point set of $\Gamma$ on $U_{i+1}$ is smooth.
By (\ref{eq:p_i^{-1}(Fix(Gamma))}), the latter is given by
\[
\Fix_{U_{i+1}}(\Gamma)= \left(\bigcup_{j=0}^2 \Fix_{U_i''}(\gamma\circ \eta_{i+1}^{j}) \right) / \left\langle \eta_{i+1}\right\rangle \ . 
\]
Since we know the Lemma already on $Y_i''$, the set $\Fix_{U_i''}(\gamma\circ \eta_{i+1}^{j})$ is smooth and $\eta_{i+1}$ acts on it.
This action is free of order three since $G_{i+1}/G_i$ acts freely on $U_i''$.
Therefore,
\[
 \Fix_{U_i''}(\gamma\circ \eta_{i+1}^{j}) / \left\langle \eta_{i+1}\right\rangle 
\]
is smooth for all $j$.
The smoothness of $\Fix_{U_{i+1}}(\Gamma)$ follows since 
\[
\Fix_{U_i''}(\gamma\circ \eta_{i+1}^{j_1}) \cap  \Fix_{U_i''}(\gamma\circ \eta_{i+1}^{j_2})=\emptyset  
\]
holds for $j_1 \nequiv j_2 \pmod 3$.
This concludes Lemma \ref{lem:Fix(Gamma)} by induction on $i$.
\end{proof}

Via diagram (\ref{eq:diag:Y_i}), we have constructed a smooth model
\[
X:=Y_c
\]
of $Y_0/\left\langle \phi_1\times \phi_2\right\rangle$.
The group $G$ acts on $X$ and the automorphism $\phi\in \Aut(X)$ which we have to construct in Proposition \ref{prop:ind.constr} is simply given by the action of $\id\times \phi_2 \in G$ on $X$.
This automorphism has order $3^c$ since this is true on the Zariski open subset $U_c\subseteq X$.
By Lemma \ref{lem:Fix(Gamma)}, the pair $(X,\phi)$ satisfies (\ref{item:coho(Fix)}); it remains to show that $(X,\phi)$ satisfies (\ref{item:h^{a,b}})--(\ref{item:Fix}).

\textbf{The cohomology of $X$.}
Using Lemma \ref{lem:Fix(Gamma)}, we are now able to read off the cohomology of $X$ from diagram (\ref{eq:diag:Y_i}).
Indeed, the cohomology of $Y_i''$ is given by the cohomology of $Y_i$ (via pullbacks) plus some classes which are introduced by blowing up $\Fix_{Y_i}(\eta_{i+1})$ on $Y_i$ and $\Fix_{Y_i'}(\eta_{i+1})$ on $Y_i'$ respectively.
By Lemma \ref{lem:Fix(Gamma)}, these blown-up loci are smooth and their $G_c$-invariant cohomology is generated by $G$-invariant algebraic classes.
Moreover, $G$ acts on each irreducible component of the blown-up locus and so $G$ acts on each irreducible component of the exceptional divisors of the blow-ups.
In particular, the corresponding divisor classes in cohomology are $G$-invariant.
It follows that the $G_c$-invariant cohomology of $Y_i''$ is given by the $G_c$-invariant cohomology of $Y_i$ plus some $G$-invariant algebraic classes. 
Also, since $\eta_{i+1}$ is contained in $G_c$, the quotient map $Y_i''\to Y_{i+1}$ induces an isomorphism on $G_c$-invariant cohomology.
It follows inductively that the $G_c$-invariant cohomology of $X$ -- which coincides with the whole cohomology of $X$ -- is given by the $G_c$-invariant cohomology of $Y_0$ plus $G$-invariant algebraic classes.

Let us now calculate the $G_c$-invariant cohomology of $Y_0$. 
For $i=1,2$, there is by assumption on $(X_i,\phi_i)$ a basis $\omega_{i1},\ldots ,\omega_{ig}$ of $H^{a_i,b_i}(X_i)$ with
\begin{equation} \label{eq:Zn1}
\phi_1^\ast(\omega_{1j})=\zeta^{-j}\omega_{1j} \ \text{and }\ \phi_2^\ast(\omega_{2j})=\zeta^{j}\omega_{2j} \ .
\end{equation}
This shows that for $j=1,\ldots ,g$, the following linearly independent $(a,b)$-classes on $Y_0$ are $G_c$-invariant:
\[
\omega_j:=\omega_{1j} \wedge \omega_{2j} \ .
\]
Since $(X_1,\phi_1^{-1})$ and $(X_2,\phi_2)$ satisfy (\ref{item:h^{a,b}}), (\ref{item:eig}) and (\ref{item:alg.classes}), it follows that apart from the above $(a,b)$-classes (and their complex conjugates), all $G_c$-invariant classes on $Y_c$ are generated by products of algebraic classes on $X_1$ and $X_2$.
These products are $G$-invariant by (\ref{item:alg.classes}).
Finally, $\phi$ acts on $\omega_j$ by multiplication with $\zeta^j$.
Altogether, we have just shown that $(X,\phi)$ satisfies (\ref{item:h^{a,b}}), (\ref{item:eig}) and (\ref{item:alg.classes}).

\textbf{Charts around $\Fix_{X}\left(\phi^{3^{c-1}}\right)$.}
By our construction, there are holomorphic charts which cover the complement of $U_c$ in $Y_c$, such that $\phi$ acts on each coordinate function by multiplication with some power of $\zeta$.
Therefore, in order to show that $(X,\phi)$ satisfies (\ref{item:Fix}), it remains to see that around points of
\[
W_c:=\Fix_{Y_c}\left(\phi^{3^{c-1}}\right) \cap U_c \ ,
\]
the same holds true.

Let us first prove that the preimage of $W_c$ under the $3^c:1$ \'etale covering $\pi: U_0\to U_c$ coincides with the following set:
\[
W_0:=\left(\left(\Fix_{X_1}\left(\phi_1^{3^{c-1}}\right)\times X_2 \right)\cup \left(X_1 \times \Fix_{X_2}\left(\phi_2^{3^{c-1}}\right) \right)\right)\cap U_0 \ .
\]
Clearly, $W_0\subseteq \pi^{-1}(W_c)$.
Conversely, suppose that $(x_1,x_2)\in \pi^{-1}(W_c)$. 
Then there exists a natural number $1\leq k\leq 3^c$ with
\[
x_1= \phi_1^k(x_1)\ \text{and}\ \phi_2^{3^{c-1}}(x_2)= \phi_2^k(x_2) \ .
\]
If $x_1$ is not fixed by $\phi_1^{3^{c-1}}$, then $3^{c-1}$ does not lie in the mod $3^c$ orbit of $k$.
That is, $k$ is divisible by $3^c$ and we deduce that $x_2$ is fixed by $\phi_2^{3^{c-1}}$. 
This shows $(x_1,x_2)\in W_0$, as we want.

Since $\pi:U_0\to U_c$ is an \'etale covering, local holomorphic charts on $U_0$ give local holomorphic charts on $U_c$.
Around each point
\[
x\in \left(\Fix_{X_1}\left(\phi_1^{3^{c-1}}\right)\times X_2 \right) \cap U_0 
\]
we may by assumptions on $(X_1,\phi_1^{-1})$ choose local holomorphic coordinates $(z_1,\ldots , z_n)$, such that $\phi_1^{-1}\times \id$ acts on each $z_j$ by multiplication with some power of $\zeta$.
Moreover, the images of $\phi_1^{-1}\times \id$ and $\id\times \phi_2$ in the quotient $G/G_c$ coincide and so the action of $\phi_1^{-1}\times \id$ on $X$ actually coincides with the automorphism $\phi$.
This shows that $(z_1,\ldots , z_n)$ give local holomorphic coordinates around $\pi(x)$ on which $\phi$ acts by multiplication with some powers of $\zeta$.

The case 
\[
x\in \left(X_1 \times \Fix_{X_2}\left(\phi_2^{3^{c-1}}\right) \right) \cap U_0 
\]
is done similarly and so we conclude that (\ref{item:Fix}) holds for $(X,\phi)$.
This finishes the proof of Proposition \ref{prop:ind.constr}.
\end{proof}

\subsection{Proof of Theorem \ref{thm:Z_n}} \label{subsec:proof:thm:Z_n}
For $a>b\geq 0$, $n\geq a+b$ and $c\geq 1$, we need to construct an $n$-dimensional smooth complex projective variety $Z_c^{a,b,n}$ whose primitive $(p,q)$-type cohomology has dimension $(3^c-1)/2$ if $p=a$ and $q=b$, and vanishes for all other $p>q$.
Suppose that we have already settled the case when $n=a+b$.
Then, for $n>a+b$, the product
\[
Z_c^{a,b,n}:=Z_c^{a,b,a+b}\times \CP^{n-a-b} 
\]
has the desired properties.
In order to prove Theorem \ref{thm:Z_n}, it therefore suffices to show that the set $\mathcal S_c^{a,b}$, defined in Section \ref{subsec:ind.constr}, is non-empty for all $a>b\geq 0$ and $c\geq 1$.
We will prove the latter by induction on $a+b$.

We put $g=(3^c-1)/2$ and consider the hyperelliptic curve $C_g$ with automorphism $\psi_g$ from Section \ref{subsec:hyperell}.
It is then straightforward to check that 
\begin{align} \label{eq:C_g_in_S}
(C_g,\psi_g)\in \mathcal S^{1,0}_c \ .
\end{align}
Indeed, it is clear that $(C_g,\psi_g)$ satisfies (\ref{item:h^{a,b}})--(\ref{item:alg.classes}) in the definition of $\mathcal S^{1,0}_c$.
Moreover, the complement of the point $\infty\in C_g$ is given by the affine curve $y^2=x^{2g+1}+1$ and $\psi_g$ acts by multiplication with a primitive $3^c$-th root of unity $\zeta$ on $x$.
For all $0\leq l\leq c-1$, the fixed point set $\Fix_{C_g}\left(\psi_g^{3^l}\right)$ is therefore given by the points $(x,y)=(0,\pm 1)$ and $\infty$.
These points are $\psi_g$-invariant and so their cohomology is generated by $\psi_g$-invariant algebraic classes, which shows that (\ref{item:coho(Fix)}) holds.
It remains to establish (\ref{item:Fix}).
That is, we need to find suitable holomorphic coordinates around the three fixed points of $\psi_g^{3^{c-1}}$.
Differentiating the affine equation $y^2=x^{2g+1}+1$ gives $2y\cdot dy=(2g+1)x^{2g}\cdot dx$.
This shows that $dx$ spans the cotangent space at $(0,\pm 1)$ and so $x$ is a local coordinate function near $(0,\pm 1)$.
The automorphism $\psi_g$ acts on this function by multiplication with $\zeta$, as we want in (\ref{item:Fix}).
In order to find a suitable coordinate function around $\infty$, we use the coordinates $(u,v)$, introduced in Section \ref{subsec:hyperell}.
In these coordinates, the curve $C_g$ is given by the equation $v^2=u+u^{2g+2}$ and $\infty$ corresponds to the point $(u,v)=(0,0)$.
Around this point, the function $v$ yields a coordinate function on which $\psi_g$ acts via multiplication with $\zeta^{g}$, see Section \ref{subsec:hyperell}.
This establishes (\ref{eq:C_g_in_S}) and hence settles the case $a+b=1$.

Let now $a>b$ with $a+b>1$.
If $b=0$, then by induction, the sets $\mathcal S_c^{1,0}$ and $\mathcal S_c^{a-1,0}$ are non-empty and so Proposition \ref{prop:ind.constr} yields an element in $\mathcal S^{a,0}_c$, as desired.
If $b\geq 1$, then  $\mathcal S_c^{a,b-1}$ is non-empty by induction. 
Also, $\mathcal S_c^{0,1}$ is non-empty since it contains $(C_g,\psi_g^{-1})$ by (\ref{eq:C_g_in_S}).
Application of Proposition \ref{prop:ind.constr} then yields an element in $\mathcal S^{a,b}_c$, as we want.
This concludes Theorem \ref{thm:Z_n}.

\begin{remark}
The variety in $\mathcal S^{a,b}_c$ which the above proof produces inductively is easily seen to be a smooth model of the quotient of $C_g^{a+b}$ by the group action of $G^1(a,b,g)$, defined in Section \ref{subsec:groupaction}.
\end{remark}

\section{Proof of Theorem \ref{thm:domination}} \label{sec:domination}
In this section we give a proof of Theorem \ref{thm:domination}, stated in the Introduction.
To begin with, we prove that $h^{1,1}$ dominates $h^{2,0}$ nontrivially in dimension two:

\begin{proposition}  \label{prop:Hodgeineq}
For a Kähler surface $X$, the following inequality holds:
\[
h^{1,1}(X)>h^{2,0}(X)\ .
\]
\end{proposition}

\begin{proof}
First, observe that for the product of $\CP^1$ with another smooth curve, $h^{2,0}$ vanishes and so the inequality trivially holds because $h^{1,1}>0$ is true for any Kähler manifold.
Since any Kähler surface of Kodaira dimension $-\infty$ is birationally equivalent to such a product \cite{surfaces}, and since $h^{2,0}$ is a birational invariant, we deduce that the asserted inequality is true in the case of Kodaira dimension $-\infty$.
Since blowing-up a point increases $h^{1,1}$ by one and leaves $h^{2,0}$ unchanged, we conclude that it suffices to prove $h^{1,1}(X)>h^{2,0}(X)$ for all minimal surfaces $X$ of non-negative Kodaira dimension.
For such a surface $X$, the Bogomolov--Miyaoka--Yau inequality
\[
c_1^2(X)\leq 3c_2(X) 
\]
holds. 
For Kodaira dimensions $0$ and $1$, this can be seen by looking at table 10 in \cite[p.\ 244]{surfaces} where all possible Chern numbers for minimal surfaces with these Kodaira dimensions are listed.
If the Kodaira dimension of $X$ is equal to $2$, that is, if $X$ is a minimal surface of general type, then the above inequality is due to Bogomolov--Miyaoka--Yau, see \cite[pp.\ 275]{surfaces}.

In order to translate the above inequality into an inequality between the Hodge numbers of $X$, we need the following identities which hold for all Kähler surfaces:
\[
c_2(X)= 2-2b_1(X)+b_2(X) \ ,
\]
\[
c_1^2(X) = 10-4b_1(X)+10h^{2,0}(X)-h^{1,1}(X)\ . 
\]
Using these, the Bogomolov-Miyaoka-Yau inequality turns out to be equivalent to
\begin{align} \label{eq:Hodgeineq}
1+h^{1,0}(X)+h^{2,0}(X)\leq h^{1,1}(X) \ .
\end{align}
This clearly implies $h^{1,1}(X)>h^{2,0}(X)$, which finishes the proof of Proposition \ref{prop:Hodgeineq}.
\end{proof}

Conversely, let us suppose that the Hodge number $h^{r,s}$ dominates $h^{p,q}$ nontrivially in dimension $n$.
That is, there are positive constants $c_1,c_2\in \R_{>0}$ such that for all $n$-dimensional smooth complex projective varieties $X$, the following holds:
\begin{align} \label{eq:dom}
c_1\cdot h^{r,s}(X)+c_2\geq h^{p,q}(X) \ .
\end{align}
By the Hodge symmetries (\ref{eq:Hsym}), we may assume $r\geq s$, $p\geq q$, $r+s\leq n$ and $1\leq p+q \leq n$.
The nontriviality of the above domination then means that (\ref{eq:dom}) does not follow from the Lefschetz conditions (\ref{eq:Lineq}).
In order to prove Theorem \ref{thm:domination}, it now remains to show $n=2$, $r=s=1$ and $p=2$.

Suppose that $r+s<n$.
Since (\ref{eq:dom}) does not follow from the Lefschetz conditions (\ref{eq:Lineq}), Theorem \ref{thm:main1} (or Corollary \ref{cor:ineqtruncatedHodgediamond} below) shows $p+q=n$.
Using the Lefschetz hyperplane theorem and the Hirzebruch--Riemann--Roch Formula, we see however that a smooth hypersurface $V_d\subseteq \CP^{n+1}$ of degree $d$ satisfies $h^{r,s}(V_d)\leq 1$, whereas $h^{p,q}(V_d)$ tends to infinity if $d$ does.
This is a contradiction and so $r+s=n$ holds.

Suppose that $r\neq s$.
Then, considering a blow-up of $\CP^n$ in sufficiently many distinct points proves $p\neq q$.
Since $p\neq q$ and $r\neq s$, we may then use certain examples from Theorem \ref{thm:Z_n} to deduce that (\ref{eq:dom}) follows from the Lefschetz conditions (\ref{eq:Lineq}).
This contradicts the nontriviality of our given domination. 
Hence, $r=s$ and in particular $n=2r$ is even.

Suppose that $p=q$.
Considering again a blow-up of $\CP^n$ in sufficiently many distinct points then proves $c_1\geq 1$ and so (\ref{eq:dom}) follows from the Lefschetz conditions.
This contradicts the nontriviality of (\ref{eq:dom}) and so it proves $p\neq q$.

Suppose that $p+q<n$.
Using high-degree hyperplane sections of $n$-dimensional examples from Theorem \ref{thm:main1}, one proves that there is a sequence of $(n-1)$-dimensional smooth complex projective varieties $(Y_j)_{j\geq 1}$ such that $h^{r-1,r-1}(Y_j)$ is bounded whereas $h^{p,q}(Y_j)$ tends to infinity if $j$ does. 
(Note that we used $p\neq q$ here.)
Since $n=2r$, we have $h^{r-1,r-1}(Y_j)=h^{r,r}(Y_j)$ by the Hodge symmetries.
Therefore, the sequence of  $n$-dimensional smooth complex projective varieties
\[
\left(Y_j\times \CP^1\right)_{j\geq 1}
\]
has bounded $h^{r,r}$ but unbounded $h^{p,q}$.
This is a contradiction and hence shows $p+q=n$.

Next, using Corollary \ref{cor:main2} from Section \ref{sec:main2}, it follows that $p=2r$ and $q=0$ holds.
By what we have shown so far we are thus left with the case where $n=2r=2s$, $p=2r$ and $q=0$.
In order to finish the proof of Theorem \ref{thm:domination}, it therefore suffices to show $r=1$.
For a contradiction, we assume that $r\geq 2$. 
By Theorem \ref{thm:Z_n} there exists a $(2r-1)$-dimensional smooth complex projective variety $Y$ with $h^{2r-1,0}(Y)=h^{0,2r-1}(Y)=1$ and $h^{p,q}(Y)=0$ for all other $p\neq q$.
Since $r\geq 2$, this implies for a smooth curve $C_g$ of genus $g$:
\[
h^{2r,0}(Y\times C_g)=g\ \ \ \text{and}\ \ \ h^{r,r}(Y\times C_g)=2\cdot h^{r-1,r-1}(Y) \ .
\]
Hence, $(Y\times C_g)_{g\geq 1}$ is a sequence of $2r$-dimensional smooth complex projective varieties such that $h^{r,r}$ is constant whereas $h^{2r,0}$ tends to infinity if $g$ does.
This is the desired contradiction and hence shows $r=1$.
This finishes the proof of Theorem \ref{thm:domination}.

\begin{remark}
One could of course strengthen Simpson's domination relation between Hodge numbers by requiring that (\ref{eq:domination}) holds for all $n$-dimensional Kähler manifolds $X$.
However, since Proposition \ref{prop:Hodgeineq} holds for all Kähler surfaces, it is immediate that Theorem \ref{thm:domination} remains true for this stronger domination relation.
\end{remark}

\section{Inequalities among Hodge and Betti numbers} \label{sec:Betti}
It is a very difficult and wide open problem to determine all universal inequalities among Hodge numbers in a fixed dimension, see \cite{Si}.
In Theorem \ref{thm:domination} we basically solved this problem for inequalities of the form (\ref{eq:domination}).\footnote{``Basically'' means that we did not determine the optimal coefficients in the universal inequality we found in dimension two.}
In this section we deduce from the main results of this paper some further progress on this problem.
We formulate our results in the category of smooth complex projective varieties, which is stronger than allowing arbitrary Kähler manifolds.

Our first result is a consequence of Theorem \ref{thm:main1}:

\begin{corollary} \label{cor:ineqtruncatedHodgediamond}
Any universal inequality among the Hodge numbers below the horizontal middle axis in (\ref{eq:diamond}) of $n$-dimensional smooth complex projective varieties is a consequence of the Lefschetz conditions (\ref{eq:Lineq}).
\end{corollary}

\begin{proof}
Assume that we are given a universal inequality between the Hodge numbers of the truncated Hodge diamond of smooth complex projective $n$-folds.
In terms of the primitive Hodge numbers $l^{p,q}$, this means that for all natural numbers $p$ and $q$ with $0< p+q<n$ there are real numbers $\lambda_{p,q}$ and a constant $C\in \R$ such that 
\begin{align} \label{eq:H1}
\sum_{\substack{0< p+q< n}} \lambda _{p,q}\cdot l^{p,q}(X)\geq C
\end{align}
holds for all smooth $n$-folds $X$.
Using the Hodge symmetries (\ref{eq:Hsym}), we may further assume that $\lambda_{p,q}=\lambda_{q,p}$ holds for all $p$ and $q$.
If we put $X=\CP^n$, then we see $C\leq 0$.
Moreover, for any natural numbers $p$ and $q$ with $0<p+q<n$, there exists by Theorem \ref{thm:main1} a smooth complex projective variety $X$ with $l^{p,q}(X)>>0$, whereas (modulo the Hodge symmetries) all remaining primitive Hodge numbers of its truncated Hodge diamond are bounded from above, by $n^3$ say.
This proves $\lambda_{p,q}\geq 0$.
That is, the universal inequality (\ref{eq:H1}) is a consequence of the Lefschetz conditions (\ref{eq:Lineq}), as we want.
\end{proof}

As an immediate consequence of the above Corollary, we note the following
\begin{corollary} \label{cor:ineq_largedim}
Any universal inequality among the Hodge numbers of smooth complex projective varieties which holds in all sufficiently large dimensions at the same time is a consequence of the Lefschetz conditions.
\end{corollary}

In the same way we deduced Corollary \ref{cor:ineqtruncatedHodgediamond} from Theorem \ref{thm:main1}, one deduces the following from Theorem \ref{thm:Z_n}:

\begin{corollary} \label{cor:thm:Z_n}
Any universal inequality among the Hodge numbers away from the vertical middle axis in (\ref{eq:diamond}) of $n$-dimensional smooth complex projective varieties is a consequence of the Lefschetz conditions (\ref{eq:Lineq}).
\end{corollary}

Corollary \ref{cor:ineqtruncatedHodgediamond} implies that in dimension $n$, the Betti numbers $b_k$ with $k\neq n$ do not satisfy any universal inequalities, other than the Lefschetz conditions
\begin{align} \label{eq:Betti:Lineq}
b_{k}\geq b_{k-2}\ \ \text{for all $k\leq n$.}
\end{align}
Using a simple construction, we improve this result now.
Indeed, the following proposition determines all universal inequalities among the Betti numbers of smooth complex projective varieties in any given dimension.

\begin{proposition} \label{prop:ineqBettinrs}
Any universal inequality among the Betti numbers $b_k$ of smooth complex projective $n$-folds is a consequence of the Lefschetz conditions (\ref{eq:Betti:Lineq}).
\end{proposition}

\begin{proof} 
By the same argument as in the proof of Corollary \ref{cor:ineqtruncatedHodgediamond}, it clearly suffices to prove the following claim.

\begin{claim} \label{claim:Betti}
Let $X$ be the product of $\CP^{n-k}$ with some smooth hypersurface $V_d \subseteq \CP^{k+1}$ of degree $d$.
Then, for $0\leq j \leq n$ with $j\neq k$, the $j$-th primitive cohomology $P^j(X)$ of $X$ has dimension $\leq 1$, whereas $\dim (P^k(X))$ tends to infinity if $d$ does.
\end{claim}

It remains to prove the claim.
For $j<k$, the Lefschetz hyperplane theorem yields:
\[
b_j(V_d)=b_{2k-j}(V_d)=b_j(\CP^{k+1}) \ .
\]
Moreover, from the Adjunction Formula we deduce that the topological Euler number $c_k(V_d)$ tends to $\pm \infty$ if $d\to \infty$.
This proves that $b_k(V_d)$ tends to infinity if $d$ does.

Using these Betti numbers of $V_d$, it is straightforward to check that
\[
X:=V_d\times \CP^{n-k}
\]
has the primitive cohomology we want.
This proves the above claim and thus finishes the proof of Proposition \ref{prop:ineqBettinrs}.
\end{proof}

\appendix
\section{Three-folds with $h^{1,1}=1$} \label{sec:3-folds}
Here we show that in dimension three, the constraints which classical Hodge theory puts on the Hodge numbers of smooth complex projective varieties are not complete.
Our result generalizes a result of Amor\'os--Biswas \cite[Prop.\ 4.3]{amoros-biswas}, asserting that there is no simply connected Kähler three-fold with $h^{2,0}=h^{1,1}=1$ and $b_3=0$. 

\begin{proposition} \label{prop:3-folds}
Let $X$ be a smooth complex projective three-fold with Hodge numbers $h^{p,q}:=h^{p,q}(X)$.
Suppose that $h^{1,1}=1$, then the following holds: 
\begin{itemize}
\item The outer Hodge numbers satisfy $h^{1,0}=0$ and $h^{2,0}<\max(h^{3,0},1)$.
\item The canonical bundle of $X$ is anti-ample if $h^{3,0}=0$, numerically trivial if $h^{3,0}=1$ and ample if $h^{3,0}>1$.
\end{itemize}
Moreover, if $h^{3,0}>1$, then $h^{2,1}< 12^6 \cdot h^{3,0}$ holds and for $h^{3,0}-h^{2,0}$ bounded from above, only finitely many deformation types of such examples exist.
\end{proposition}

Proposition \ref{prop:3-folds} nicely compares to the examples in Theorem \ref{thm:main3}, where we have constructed three-folds $X$ with $h^{1,1}(X)=1$ such that $h^{2,0}(X)$ is equal to any given natural number.

Before we can prove Proposition \ref{prop:3-folds}, let us show the following general result.

\begin{lemma} \label{lem:bilinrel}
Let $X$ be a Kähler manifold of dimension $n$ and let $k$ be an odd natural number with $2k\leq n$ 
such that $h^{k,k}(X)=1$. 
Then, $b_j(X)=0$ for all odd $j\leq k$.
\end{lemma} 

\begin{proof}
Let $\omega$ denote the Kähler class of $X$.
For a contradiction, suppose that the assumptions of the Proposition hold and that additionally $b_j(X)\neq0$ for some odd $j\leq k$.
We may assume that $j$ is minimal with this property.
Then all $j$-th cohomology is primitive and we pick some non-zero primitive $(p,q)$-cohomology class $\alpha$ with $p+q=j$.
Since $h^{k,k}(X)=1$ and since $2k\leq n$, the Lefschetz conditions (\ref{eq:Lineq}) imply that $H^{j,j}(X)$ is spanned by $\omega^j$.
Thus, by the Hodge--Riemann bilinear relations:
\[
\alpha \wedge \overline{\alpha} =\lambda \cdot \omega^j
\]
for some $\lambda \in \C - \left\{0\right\}$.
Since $2j\leq 2k\leq n$, we have $\omega^{2j}\neq 0$.
As $\alpha$ is of odd degree, this is a contradiction to the above equation and hence establishes the Lemma.
\end{proof}

\begin{proof}[Proof of Proposition \ref{prop:3-folds}]
Let $X$ be a smooth complex projective three-fold with 
\[
h^{1,1}(X)=1\ .
\]
The Riemann--Roch Formula in dimension three says
\begin{align} \label{eq:RR}
c_1c_2(X)=24\chi(X,\mathcal O_X) \ .
\end{align}

By Lemma \ref{lem:bilinrel} we have $h^{1,0}(X)=h^{0,1}(X)=0$.
From the exponential sequence, it therefore follows that $X$ has Picard number one and hence the canonical class $K_X$ of $X$ is either ample, anti-ample or numerically trivial.

If $-K_X$ is ample, then $h^{2,0}$ and $h^{3,0}$ vanish.

If $K_X$ is numerically trivial, than (\ref{eq:RR}) shows $1+h^{2,0}=h^{3,0}$.
Since numerically trivial line bundles have at most one nontrivial section, we deduce $h^{2,0}=0$ and $h^{3,0}=1$.

If $K_X$ is ample, then Yau's inequality holds \cite{yau}:
\begin{align} \label{eq:yau}
c_1c_2(X)\leq \frac{3}{8}c_1^3(X) \ .
\end{align}
Together with (\ref{eq:RR}), this implies 
\begin{align} \label{eq:yau+RR}
\chi(X,\mathcal O_X)\leq \frac{1}{64}c_1^3(X)<0 \ .
\end{align}
Thus, $-c_1^3(X)$ can be bounded from above in terms of $h^{3,0}-h^{2,0}$ and hence Koll\'ar--Matsusaka's theorem \cite[pp.\ 239]{laz2} yields that only finitely many deformation types of three-folds with $h^{1,1}=1$, $h^{3,0}>1$ and $h^{3,0}-h^{2,0}$ from above bounded exist.
Furthermore, (\ref{eq:yau+RR}) shows that $1+h^{2,0}<h^{3,0}$ holds for any such three-fold.

Altogether, this proves firstly $h^{2,0}<\max(h^{3,0},1)$, and secondly that $K_X$ is anti-ample if $h^{3,0}=0$, it is numerically trivial if $h^{3,0}=1$ and it is ample if $h^{3,0}>1$.

Finally, let us assume that $h^{3,0}>1$ or $h^{2,0}>0$.
Then $K_X$ is ample and so Fujita's conjecture predicts that $6\cdot K_X$ is very ample, cf.\ \cite[p.\ 252]{laz2}.
Although this conjecture is still open, Lee proves in \cite{lee} that $10\cdot K_X$ is very ample.
Thus, the following argument due to Catanese--Schneider \cite{catanese-schneider} applies:
Firstly, the linear series $|10\cdot K_X|$ embeds $X$ into some $\CP^N$ and hence $\Omega_X(20\cdot K_X)$ is a quotient of $\Omega_{\CP^N}(2)$ restricted to $X$.
Since the latter is globally generated, it is nef and hence $\Omega_X(20\cdot K_X)$ is nef.
Secondly, by \cite[Cor.\ 2.6]{d-p-s}, any Chern number of a nef bundle $F$ on an $n$-dimensional smooth complex projective variety $X$ is bounded from above by $c_1^n(F)$.
In our situation, this yields
\begin{align} \label{eq:d-p-s}
c_3(\Omega^1_X(20\cdot K_X))\leq c_1^3(\Omega^1_X(20\cdot K_X)) \ .
\end{align}
A standard computation gives
\[
c_3(\Omega^1_X(20\cdot K_X))= -8400 \cdot c_1^3(X)-20 \cdot c_1c_2(X)-c_3(X)
\]
and
\[
c_1^3(\Omega^1_X(20\cdot K_X))=-61^3\cdot c_1^3(X) \ .
\]
Together with Yau's inequality (\ref{eq:yau}), this yields in (\ref{eq:d-p-s})
\begin{equation} \label{eq:c12<c3}
1\:748\:588 \cdot c_1c_2(X)\leq 3 \cdot c_3(X)  \ .
\end{equation}
By the Riemann--Roch formula, this inequality is in fact one between the Hodge numbers of three-folds with ample canonical bundle.
In our case, $h^{1,1}=1$ and $h^{1,0}=0$ yield:
\[
6994346 + 6994346 \cdot h^{2,0} + 3\cdot h^{2,1} \leq 6\:994\:349 \cdot h^{3,0} \ .
\]
Thus, a rough estimation yields
\[
h^{2,1}< 12^6\cdot h^{3,0} \ .
\]
This concludes the proof of the Proposition.
\end{proof}
\begin{remark}
Instead of using \cite{d-p-s}, but still relying on \cite{lee}, Chang--Lopez prove in \cite{chang-lopez} that there is a computable constant $C>0$ such that $C\cdot c_1c_2(X)\leq c_3(X)$ holds for all three-folds $X$ with ample canonical bundle.
Computing $C$ explicitly shows that it is about four times smaller then the analogous constant which appears in (\ref{eq:c12<c3}). 
However, since the explicit extraction of $C$ is slightly tedious and since this constant is still far from being realistic, we did not try to carry this out here.
\end{remark}

Using Proposition \ref{prop:3-folds} together with the classification of Fano three-folds \cite[p. 215]{fano3folds}, we obtain the following classification of Hodge diamonds of three-folds with $h^{1,1}=1$ and $h^{3,0}=0$. 
\begin{corollary}
Let $h^{p,q}$ be the Hodge numbers of a smooth complex projective three-fold with $h^{1,1}=1$ and $h^{3,0}=0$.
Then $h^{1,0}$ and $h^{2,0}$ vanish, and for $h^{2,1}$ precisely one of the following values occurs: 
\[
h^{2,1}\in \left\{0,2,3,5,7,10,14,20,21,30,52\right\} \ .
\]
\end{corollary}

\section{Four-folds with $h^{1,1}=1$} \label{sec:4-folds} 
Here we show that in dimension four, the constraints which classical Hodge theory puts on the Hodge numbers of smooth complex projective varieties are not complete.

\begin{proposition} \label{prop:4-folds}
Let $X$ be a smooth complex projective four-fold with Hodge numbers $h^{p,q}:=h^{p,q}(X)$.
If $h^{1,1}=1$, then $h^{1,0}=0$ and for bounded $h^{2,0}$, $h^{4,0}$ and $h^{2,2}$, only finitely many values for $h^{3,0}$, $h^{2,1}$ and $h^{3,1}$ occur.
\end{proposition}

Since Kähler manifolds with $b_2=1$ are projective, Proposition \ref{prop:4-folds} implies immediately that even for the Betti numbers of Kähler manifolds, the known constraints are not complete.

\begin{corollary} \label{cor:4-folds}
Let $X$ be a Kähler four-fold with $b_2(X)=1$. 
Then $b_3(X)$ is bounded in terms of $b_4(X)$.
\end{corollary}

\begin{proof}[Proof of Proposition \ref{prop:4-folds}]
Let $X$ be a smooth complex projective four-fold with Hodge numbers $h^{p,q}$ and Chern numbers $c_1^4,c_1^2c_2,\ldots ,c_4$.
Suppose that $h^{1,1}=1$ and that $h^{2,0}$, $h^{4,0}$ and $h^{2,2}$ are bounded.
Then Lemma \ref{lem:bilinrel} shows $h^{1,0}=0$.
Moreover, we have:

\begin{lemma} \label{lem:h11=1estimate}
The following inequality holds:
\[
224+228h^{2,0}-224h^{3,0}+h^{2,2}-2h^{3,1}+226h^{4,0}\geq \frac{1}{3}\cdot \left(4c_1^2c_2-c_1^4\right) \ .
\]
\end{lemma}

\begin{proof}
Since $h^{1,1}=1$, we see that $c_2(X)=\lambda\cdot \omega^2 +\alpha$ where $\alpha$ is a primitive $(2,2)$-class and $\omega$ the Kähler class on $X$. 
Since $\omega$ and $c_2(X)$ are real cohomology classes, we obtain
\[
\alpha-\overline{\alpha}=-(\lambda-\overline{\lambda})\cdot \omega^2 \ .
\]
In this equation, the left-hand side is primitive.
However, no nonzero multiple of $\omega^2$ is primitive and we conclude $\alpha=\overline{\alpha}$ and $\lambda\in \R$.
Thus, by the Hodge--Riemann bilinear relations:
\[
\int_X \alpha\wedge \overline{\alpha}=\int_X\alpha^2 \geq 0 \ .
\]
This implies, since $\alpha\wedge \omega =0$ and $\lambda \in \R$:
\begin{align} \label{eq:c_2^2}
\int_Xc_2(X)^2=\int_X\left(\lambda^2\omega^4+2\lambda\cdot \omega^2\wedge\alpha+\alpha^2\right) \geq 0 \ .
\end{align}

Let us now use the following formula, due to Libgober--Wood \cite{libgober}:
\begin{align} \label{eq:libgober}
c_1c_3=12\chi^2-36\chi^3+72\chi^4-14c_4 \ ,
\end{align}
where $\chi^p=\sum_q(-1)^qh^{p,q}$.
By the Riemann--Roch theorem and (\ref{eq:c_2^2}), we also have
\[
\chi^4=\frac{1}{720}\left(-c_4+c_1c_3+3c_2^2+4c_1^2c_2-c_1^4\right) \geq \frac{1}{720}\left(-c_4+c_1c_3+4c_1^2c_2-c_1^4\right) \ .
\]
Using Libgober--Wood's expression for $c_1c_3$, this reads:
\[
\chi^4\geq \frac{1}{720}\left(-15c_4+12\chi^2-36\chi^3+72\chi^4+4c_1^2c_2-c_1^4\right) \ .
\]
Finally, expressing the topological Euler characteristic $c_4$ as well as all the $\chi^p$'s in terms of Hodge numbers, one obtains the inequality, claimed in the Lemma.
\end{proof}

Since $h^{1,1}=1$ and $h^{1,0}=0$, we see that $X$ has Picard number one.
Thus, the canonical class $K_X$ is either anti-ample, numerically trivial or ample.

In fixed dimension, there are only finitely many deformation types of smooth complex projective varieties with anti-ample canonical class \cite{fano}.
Since deformation equivalent varieties have the same Hodge numbers \cite[p.\ 235]{voisin1}, the Proposition is true in this case.

If $K_X$ is numerically trivial, then Lemma \ref{lem:h11=1estimate} implies that $h^{3,0}$ and $h^{3,1}$ are bounded.
Moreover, Libgober--Wood's formula (\ref{eq:libgober}) shows:
\[
52+40h^{2,0}-4h^{2,1}-2h^{2,2}-52h^{3,0}+8h^{3,1}+44h^{4,0}=0 \ .
\]
Since we already know that apart from $h^{2,1}$ all Hodge numbers in the above identity are bounded, it follows that $h^{2,1}$ is bounded as well.

Finally, it remains to deal with the case where $K_X$ is ample.
Here, Yau's inequality \cite{yau} holds:
\[
c_1^2c_2\geq \frac{2}{5} c_1^4 \ .
\]
Using this, we obtain from Lemma \ref{lem:h11=1estimate}:
\[
224+228h^{2,0}-224h^{3,0}+h^{2,2}-2h^{1,3}+226h^{4,0}\geq \frac{1}{5}c_1^4 \ .
\]
Since $h^{2,0}$, $h^{4,0}$ and $h^{2,2}$ are bounded, we deduce that $c_1^4$ is bounded from above.
Thus, Koll\'ar--Matsusaka's Theorem \cite[pp.\ 239]{laz2} implies that only finitely many deformation types of such four-folds exist.
As in the case of anti-ample canonical class it follows that $h^{3,0}$, $h^{2,1}$ and $h^{3,1}$ are bounded.
This concludes the proof of Proposition \ref{prop:4-folds}. 
\end{proof}

\section*{Acknowledgment}
Thanks to Burt Totaro for drawing my attention to the construction problem and for informing me about the results in \cite{catanese-schneider,chang-lopez}.
I am grateful to my advisor Daniel Huybrechts for stimulating discussions and to Ciaran Meachan, Dieter Kotschick and Burt Totaro for useful comments.
I would also like to warmly thank two anonymous referees; this paper benefited significantly from their careful reading and suggestions. 
The author is supported by an IMPRS Scholarship of the Max--Planck--Society.

\end{document}